\documentclass[11pt]{article}
\usepackage{amssymb}
\usepackage{mathrsfs}
\usepackage{amsfonts}
\usepackage{amsmath,amsthm,amsfonts,amssymb,amscd}
\usepackage{tikz}

\allowdisplaybreaks[4]
\newtheorem {Lemma}{Lemma}[section]

\newtheorem {Theorem} [Lemma]{Theorem}
\newtheorem {Corollary}[Lemma]{Corollary}
\newenvironment {Proof} {\noindent {\bf Proof.}}{\quad $\square$\par\vspace{3mm}}

\begin{document}

\title{Erd\H{o}s-Gallai Stability Theorem for Linear Forests \thanks{ This work is supported by  the Joint NSFC-ISF Research Program (jointly funded by
the National Natural Science Foundation of China and the Israel Science Foundation (No. 11561141001)),
the National Natural Science Foundation of China (Nos.11531001 and 11271256),  Innovation Program of Shanghai Municipal
 Education Commission (No. 14ZZ016) and Specialized Research Fund for the Doctoral Program of Higher Education (No.20130073110075).
\newline \indent $^{\dagger}$Corresponding author:
Xiao-Dong Zhang (Email: xiaodong@sjtu.edu.cn),
}
}

\author{ Ming-Zhu Chen and Xiao-Dong Zhang%\footnote{Corresponding author. E-mail: xiaodong@sjtu.edu.cn}
\\
School of Mathematical Science,  MOE-LSC, SHL-MAC \\
Shanghai Jiao Tong University,
Shanghai 200240, P. R. China}

\date{}
\maketitle

\begin{abstract}

The Erd\H{o}s-Gallai Theorem states that  every graph of average degree more than $l-2$ contains a path of order $l$ for $l\ge 2$.
In this paper, we obtain a stability version of the Erd\H{o}s-Gallai Theorem in terms of minimum degree.  Let $G$ be a connected graph of order $n$ and
 $F=(\bigcup_{i=1}^kP_{2a_i})\bigcup(\bigcup_{i=1}^lP_{2b_i+1})$ be $k+l$ disjoint paths of order $2a_1, \ldots, 2a_{k}, 2b_1+1, \ldots, 2b_l+1,$ respectively,
 where $k\ge 0$, $0\le l\le 2$, and $k+l\geq 2$. If the minimum degree $\delta(G)\ge \sum_{i=1}^ka_i+\sum_{i=1}^lb_i-1$, then  $F\subseteq G$
 except several classes of graphs for  sufficiently large  $n$, which extends and strengths the results of Ali and Staton for an even path and
Yuan and Nikiforov for an odd path.
\end{abstract}

\section{Introduction}
\subsection{Notation}

Let $G$ be a  finite simple undirected graph with vertex set $V(G)$ and edge set $E(G)$.
Denote  $e(G)$ by the number of edges of $G$.
 For $v\in V(G)$,  the \emph{neighborhood} $N_G(v)$ of $v$  is $\{u: uv\in E(G)\}$ and the \emph{degree} $d_G(v)$ of $v$  is $|N_G(v)|$.
%If there is no ambiguity, then we write $N(v)$ and $d(v)$ for $N_G(v)$ and $d_G(v)$, respectively.
For a subgraph $H$  of $G$,  the \emph{$H$-neighborhood}  $N_H(v)$ of $v$  is $N_G(v)\bigcap V(H)$ and the
\emph{$H$-degree} $d_H(v)$ of $v$  is $|N_H(v)|$.
 Denote  $\delta(G)$ by   minimum degree of $G$.
Denote  $P_n$, $K_n$,  and $K_{m,n}$  by a path of order $n$,  a complete graph of order $n$, and
 a complete bipartite graph with  partition of size $m$ and $n$, respectively.
For two disjoint graphs $G$ and $H$,  denote $G\bigcup H$ and $G\bigvee H$ by the disjoint union of $G$ and $H$, and the join of $G$ and $H$ which is  obtained from $G\bigcup H$ by joining every vertex of $G$ to every vertex of $H$, respectively.
Moreover,  $kG$  denotes a graph consisting of $k$ disjoint copies of  $G$ and  $\overline{G}$ denotes the complement of $G$.
%For a subgraph $H$ of $G$,  $G-V(H)$ denotes the graph obtained from $G$ by deleting all vertices  and incident edges of $H$. In particular,
% write $G-v$ for $G-\{v\}$.
For $X\subseteq V(G)$,  $G[X]$ denotes the graph induced by $X$ and write $E(X)$ for $E(G[X])$.
For $X, Y\subseteq V(G)$,  $e(X,Y)$ denotes the number of the  edges of  $G$ with one end vertex in $X$ and the other  in $Y$.
For two graphs $G$ and $H$, write $H\subseteq G$ if $G$ contains $H$ as a subgraph, and $H\nsubseteq G$  otherwise.
An \emph{odd (even) path} is a path of odd (even) order.
A graph is  \emph{connected} if there is a path between every pair of vertices. A connected graph $G$ is   \emph{$k$-connected} if $|V(G)|>k$ and $G-X$ is connected for every set $X\subseteq V(G)$ with $|X|<k$.
%For $V_1, V_2\subseteq V(G)$, denote by $E(V_1,V_2)$  the set of  all edges in $G$ with one vertex in $V_1$
% and the other vertex in $V_2$ and  $e(V_1,V_2)=|E(V_1,V_2)|$.
A graph $G$ is \emph{$F$-free} if it  does not contain $F$ as a subgraph.
A \emph{star forest} is a forest whose components are stars and a \emph{linear forest}
is a forest whose  components are paths.
A \emph{cut vertex (edge)} of a graph $G$ is a vertex (edge) whose removal increases the number of components of $G$.
A \emph{block} of is a maximal connected subgraph without any cut vertex.
 An \emph{end block} of  $G$ is block that contains precisely a  cut vertex of $G$.
For two end blocks $B_1$ and $B_2$ of a graph $G$,
 denote $p(B_1,B_2)$   by  the order of a longest path between  the unique cut vertex of $G$ in $V(B_1)$ and  the unique cut vertex of $G$ in $V(B_2)$.
For other notations not defined here, readers are referred to \cite{West}.
\subsection{History and known results}

The study of Tur\'{a}n-type problems is the center of  extremal graph theory.
%   For
%Among Tur\'{a}n-type problems for graphs, one of the questions which draw much more attention  is to determine the
%maximum number of edges in a graph of order $n$ which does not  contain  a path  of a specified order.
%This question was answered by
Erd\H{o}s and Gallai in \cite{EG} proved the following key result, which opens a new subject for degenerate graphs:
\begin{Theorem}\label{Thm1.1}\cite{EG}
Let $G$ be a   graph of order $n$.   If $e(G)>\frac{(l-2)n}{2}$, where $l\geq2$,  then $P_l\subseteq G$.
\end{Theorem}
 Let $H_{n,l,a}:=K_a\bigvee (K_{l-2a}\bigcup \overline{K}_{n-l+a})$ with
 $h(n,l,a):=e(H_{n,l,a})=\binom{l-a}{2}+a(n-l+a)$ for $a\le \lfloor\frac{l}{2}\rfloor.$
Recently,  F\"{u}redi,  Kostochka, and Verstra\"{e}te  \cite{FKV} obtained a stability theorem for paths.
\begin{Theorem}\label{Thm1.1}\cite{FKV}
Let $G$ be a connected graph of order $n$, where  $t\geq2$ and $n\geq 3t-1$.
If $e(G)> h(n+1,l+1,t-1)-n$, where $l\in\{2t,2t+1\} $, then $P_l\subseteq G$, unless one of following holds:\\
(i). $l=2t$, $l\neq6$, and $G\subseteq H_{n,l,t-1}$;\\
(ii). $l=2t+1$ or $l=6$, and $G-A$ is a star forest for some $A\subseteq V(G)$ of size at most $l-1$.
\end{Theorem}
Another question in Tur\'{a}n-type problems which has received  extensive attention  is how large
minimum degree of a graph $G$ is to guarantee  $P_l\subseteq G$.
Erd\H{o}s and Gallai \cite{EG} and  Andrasfai \cite{A} proved the following result for an odd path.
\begin{Theorem}\label{Thm1.3}\cite{EG,A}
Let $G$ be a connected graph of order $n$, where $n\geq 2h+3\geq3$. If $\delta (G)\geq h+1$,
then $P_{2h+3}\subseteq G$.
\end{Theorem}
Let $S_{n,h}:=K_h\bigvee \overline{K}_{n-h}$  and  $S^+_{n,h}$ be a graph obtained by adding an edge to $S_{n,h}$,
i.e., $S^+_{n,h}=K_h\bigvee (K_2\bigcup \overline{K}_{n-h-2})$ (see Fig.1).

\vspace{4mm}
\begin{center}
\begin{tikzpicture}[scale=1.4]
\draw (0,0) ellipse (1cm and 0.5cm);\draw (0,-1.5) ellipse (1cm and 0.5cm);
\path (-0.6,-0.2)  coordinate (P1);\path (0.6,-0.2)  coordinate (P2);
\path (-0.6,-1.3)  coordinate (P3);\path (0.6,-1.3)  coordinate (P4);
\path (-0.2,-0.2)  coordinate (Q1);\path (0,-0.2)  coordinate (Q2);\path (0.2,-0.2)  coordinate (Q3);
\path (-0.2,-1.3)  coordinate (Q4);\path (0,-1.3)  coordinate (Q5);\path (0.2,-1.3)  coordinate (Q6);
\foreach \i in {1,2,3,4}
{\fill (P\i) circle (1pt);}
\foreach \i in {1,...,6}
{\fill (Q\i) circle (0.6pt);}
\foreach \i in {1,2}
{
\foreach \j in {3,4}
{
\draw (P\i) -- (P\j);
}
}
\path (0,0.2) node () {\scriptsize{$K_h$}};\path (0,-1.7) node () {\scriptsize{$\overline{K}_{n-h}$}};
\path (0,-2.4) node () {$S_{n,h}$};

\draw (3,0) ellipse (1cm and 0.5cm);\draw (3,-1.5) ellipse (1cm and 0.5cm);
\path (2.4,-0.2)  coordinate (P1);\path (3.6,-0.2)  coordinate (P2);
\path (2.4,-1.3)  coordinate (P3);\path (3.6,-1.3)  coordinate (P4);
\path (2.8,-1.3)  coordinate (P5); \path (3,-1.3)  coordinate (P6);
\path (2.8,-0.2)  coordinate (Q1);\path (3,-0.2)  coordinate (Q2);\path (3.2,-0.2)  coordinate (Q3);
\path (3.2,-1.3)  coordinate (Q4);\path (3.3,-1.3)  coordinate (Q5);\path (3.4,-1.3)  coordinate (Q6);
\foreach \i in {1,...,6}
{\fill (P\i) circle (1pt);}
\foreach \i in {1,...,6}
{\fill (Q\i) circle (0.6pt);}
\foreach \i in {1,2}
{
\foreach \j in {3,...,6}
{
\draw (P\i) -- (P\j);
}
}
\draw (P3) -- (P5);
\path (3,0.2) node () {\scriptsize{$K_h$}};\path (3,-1.7) node () {\scriptsize{$ K_2\bigcup\overline{ K}_{n-h-2}$}};
\path (3,-2.4) node () {$S^+_{n,h}$};

\path (2,-3) node () { Fig.1. Graphs $S_{n,h}$ and $S_{n,h}^+$.};
\end{tikzpicture}
\end{center}

Given $t_1,t_2\geq 0$, let $L_{t_1,t_2,h,h+1}:=K_1\bigvee (t_1K_h\bigcup t_2K_{h+1})$ (see Fig.2). In particular, write $ L_{t,0,h,h+1}:=L_{t,h}$ (see Fig.2). The \emph{center} of  $L_{t_1,t_2,h,h+1}$ is the vertex of maximum degree in $L_{t_1,t_2,h,h+1}$.

\vspace{4mm}
\begin{center}
\begin{tikzpicture}[scale=1.4]
\path (1,0)  coordinate (P1);\path (0,0.3)  coordinate (P2);\path (0.5,1)  coordinate (P3);
\path (1.5,1)  coordinate (P4);\path (2,0.3)  coordinate (P5);
\path (0,-0.3)  coordinate (P6);\path (0.5,-1)  coordinate (P7);
\path (1.5,-1)  coordinate (P8);\path (2,-0.3)  coordinate (P9);
\path (0.8,0.7)  coordinate (Q1);\path (1,0.7)  coordinate (Q2);\path (1.2,0.7)  coordinate (Q3);
\path (0.8,-0.7)  coordinate (Q4);\path (1,-0.7)  coordinate (Q5);\path (1.2,-0.7)  coordinate (Q6);
\foreach \i in {1,...,9}
{\fill (P\i) circle (1pt);}
\foreach \i in {1,...,6}
{\fill (Q\i) circle (0.5pt);}
\draw (P1)--(P2)--(P3)--cycle;
\draw (P1)--(P4)--(P5)--cycle;
\draw (P1)--(P6)--(P7)--cycle;
\draw (P1)--(P8)--(P9)--cycle;
\path (0.5,0.4) node () {\scriptsize{$K_{h+1}$}};\path (1.6,0.4) node () {\scriptsize{$K_{h+1}$}};
\path (0.5,-0.4) node () {\scriptsize{$K_{h+2}$}};\path (1.6,-0.4) node () {\scriptsize{$K_{h+2}$}};
\path (1,0.9) node () {$t_1$};\path (1,-0.9) node () {$t_2$};
\path (1,-1.5) node () {$L_{t_1,t_2,h,h+1}$};

\path (5,-0.5)  coordinate (P1);\path (4,-0.2)  coordinate (P2);\path (4.5,0.5)  coordinate (P3);
\path (5.5,0.5)  coordinate (P4);\path (6,-0.2)  coordinate (P5);
\path (4.8,0.2)  coordinate (Q1);\path (5,0.2)  coordinate (Q2);\path (5.2,0.2)  coordinate (Q3);
\foreach \i in {1,...,5}
{\fill (P\i) circle (1pt);}
\foreach \i in {1,...,3}
{\fill (Q\i) circle (0.5pt);}
\draw (P1)--(P2)--(P3)--cycle;
\draw (P1)--(P4)--(P5)--cycle;
\path (4.5,-0.1) node () {\scriptsize{$K_{h+1}$}};\path (5.6,-0.1) node () {\scriptsize{$K_{h+1}$}};
\path (5,0.4) node () {$t$};
\path (5,-1.5) node () {$L_{t,h}$};
\path (3,-2.3) node () {Fig.2. Graphs $L_{t_1,t_2,h,h+1}$ and $L_{t,h}$.};
\end{tikzpicture}
\end{center}

For an even path, the question was answered by Ali and Staton \cite{AS}.
\begin{Theorem}\cite{AS}
Let $G$ be a  connected graph of order $n$, where $n\geq 2h+2$.
If $\delta (G)\geq h\geq1$, then
$P_{2h+2}\subseteq G$ unless either $G\subseteq S_{n,h}$ or  $G=L_{t,h}$, where $n = th+1$.
\end{Theorem}

\underline{}Recently, Yuan and Nikiforov \cite{NY} gave a stability version of  Theorem~\ref{Thm1.3}.
\begin{Theorem}\cite{NY}
Let $G$ be a connected graph of order $n$, where $n\geq 2h+3$. If $\delta (G)\geq h\geq2$, then
$P_{2h+3}\subseteq G$, unless unless one of following holds:\\
(i). $G\subseteq S_{n,h}^+$;\\
(ii). $G=L_{t,h}$, where $n= th+1$;\\
(iii). $G \subseteq  L_{t,h,1,h+1}$, where $n =(t+1)h+2 $;\\
(iv). $G$ is obtained by joining the centers of two disjoint graphs $L_{s,h}$ and $L_{t,h}$,
 where $n = (s +t)h+2$.
\end{Theorem}

    Bushaw and Kettle \cite{BK} proved the maximum number of edges in a $kP_l$-free graph  of  sufficiently large order $n$ for $l\geq3$ and determined all the extremal graphs.
Recently,    Yuan and Zhang \cite{Yuan} extended their result for  $l=3$ and all possible  $n$ and $k$.
Lidick\'{y}, Liu, and Palmer \cite{LLP} extended  Bushaw-Kettle result for linear forests. Their main result can be stated as follows.

%An natural extension of  the Tur\'{a}n-type problems of paths is to determine Tur\'{a}n-type problems for liner forests.

%\cite{LLP} determined the maximum number of edges in a graph of order $n$ which does not  contain  a specified linear forest:

\begin{Theorem}\label{Th1.6}\cite{LLP}
 Let  $F=(\bigcup_{i=1}^k P_{2a_i})\bigcup (\bigcup_{i=1}^lP_{2b_i+1})$  and
 $ h=\sum \limits_{i=1}^k a_i+\sum \limits_{i=1}^l b_i-1$, where   $k+l\geq2$,
$a_1\geq \cdots \geq a_{k}\geq1$, and  $b_i\geq1$ for $1\leq i\leq l$. Let $G$ be  an $F$-free  graph  of  sufficiently large order $n$.\\
(i). If $k\geq1$, then $e(G)\leq e(S_{n,h})$ with equality if and only if $G=S_{n,h}$;\\
(ii). If $k=0$ and there is at least one $b_i$ such that $b_i>1$, where $1\leq i \leq l$,  then $e(G)\leq e(S^+_{n,h})$ with equality if and only if $G=S^+_{n,h}$.\\
\end{Theorem}

It is easy to see that we have the following result from Theorem~\ref{Th1.6}.
\begin{Corollary}\label{Cor}
 Let $F=(\bigcup_{i=1}^ k P_{2a_i})\bigcup (\bigcup_{ i=1}^ lP_{2b_i+1})$ and
 $ h=\sum \limits_{i=1}^k a_i+\sum \limits_{i=1}^l b_i-1$, where    $k+l\geq2$,
$a_1\geq \cdots \geq a_{k}\geq1$, and $b_i\geq1$  for $1\leq i\leq l$.  Let $G$ be a  graph  of sufficiently large order $n$.\\
 (i). If $k\geq1$ and $\delta>2h-\frac{h^2+h}{n}$, then $F\subseteq G$;\\
 (ii). If $k=0$,  $\delta>2h-\frac{h^2+h-4}{n}$, and  there is at least one $b_i$ such that $b_i>1$, where $1\leq i\leq l$,   then $F\subseteq G$.
\end{Corollary}

The related results can be referred to \cite{CZ2018} and references therein.
%Moreover,     Bushaw and Kettle \cite{BK} proved the maximum number of edges in a graph  of order $n$  which does not contain $kP_3$ as
% a subgraph for  $n\geq7k$  and characterized all extremal graphs.
%Recently,    Yuan and Zhang \cite{Yuan} extended their result for all possible $n$ and $k$.

\subsection{Our main results}

In this paper, we obtain  stability versions of Corollary~\ref{Cor} %minimum degree condition
for a graph which does not contain  a  linear forest with at most two odd paths.
The main results in this paper are Theorems~\ref{Thm1}-\ref{Thm2}.

%Theorem~\ref{Thm1} gives  stability theorem for linear forests where all components are even paths.
%Theorem~\ref{Thm4} gives a stability theorem for linear forests where only one component is  odd path.
%Theorem~\ref{Thm2}  and  Theorem~\ref{Thm5} give  stability theorems for  linear forests  where two components are  odd paths.

\begin{Theorem}\label{Thm1}
Let $F=\bigcup_{i=1}^k P_{2a_i}$ and $ h=\sum _{i=1}^k a_i-1\geq1$, where  $k\geq2$ and
$a_1\geq \cdots \geq a_{k}\geq1$. Let $G$ be a  connected graph of order $n$, where $n\geq2h+2$. If  $\delta(G)\geq h$,
 then $F\subseteq G$, unless one of the following holds:\\
 (i). $G\subseteq S_{n,h}$;\\
  (ii). $F=2P_{2a_1}$ and $G=L_{t,h}$, where $n=th+1$.
\end{Theorem}

{\bf Remark 1.} The minimum degree condition  is best possible.  For example,
  let   $F=2P_{2a_1}$ and  $h=2a_1-1$.  If $G= L_{\frac{n-1}{h-1},h-1}$,  where $n=h(h-1)q+1$ with positive integer $q$,
then  $\delta(G)=h-1$, $F\nsubseteq G$, $G\nsubseteq S_{n,h}$,
and $G\ncong L_{\frac{n-1}{h},h}$. So Theorem~\ref{Thm1} does not hold for $\delta(G)=h-1$.
For another example, let    $F=P_2\bigcup 2P_4$ and $h=4$.  If $G=L_{\frac{n-1}3,3}$£¬ where $n=12q+1$ with positive integer $q$,  then $\delta(G)=3$, $F\nsubseteq G$, and $G\nsubseteq S_{n,4}$.
%and $G\ncong L_{\frac{n-1}{4},4}$.
 So Theorem~\ref{Thm1} does not hold for  $\delta(G)=h-1$.
Therefore the condition $\delta(G)\geq h$  in Theorem~\ref{Thm1} can not be weaken and  is best possible.

%Note that $\delta(L_{6,2})=2$, $2P_4 \nsubseteq L_{6,2}$, $L_{6,2}\nsubseteq S_{13,3}$ and $L_{6,2}\ncong L_{4,3}$.
%Thus the claim of Theorem~\ref{Thm1} does not hold for $G=L_{6,2}$ and $F=2P_4$.
%Therefore the condition $\delta(G)\geq h$  in Theorem~\ref{Thm1}  is best possible for $F=2P_{2a_1}$.

\begin{Theorem}\label{Thm4}
Let $F=(\bigcup_{i=1}^k P_{2a_i})\bigcup P_{2b_1+1}$ and  $ h=\sum _{i=1}^k a_i+b_1-1\geq1$, where  $k\geq1$,
$a_1\geq \cdots \geq a_{k}\geq1$, and $b_1\geq1$. Let $G$ be a connected graph of order $n$, where $n\geq2h+3$. If  $\delta(G)\geq h$,
then $F\subseteq G$, unless one of the following holds:\\
 (i). $G\subseteq S_{n,h}$;\\
 (ii).  $F=P_6\bigcup P_3$ and $G\subseteq K_2\bigvee \frac{n-2}{2}K_2$, where $n$ is even;\\
  (iii). $F\in \{P_{2b_1}\bigcup P_{2b_1+1}, P_{2b_1+2}\bigcup P_{2b_1+1}\}$ and $G=L_{t,h}$, where  $n=th+1$.

\end{Theorem}
%By similar argument to Theorem~\ref{Thm1}, the condition $\delta(G)\geq h$ in Theorem~\ref{Thm4} is best possible.
{\bf Remark 2}. The minimum degree condition  is best possible. For example,
 let  $F=P_{2b_1}\bigcup P_{2b_1+1}$ and $h=2b_1$. If $G=L_{\frac{n-1}{h-1},h-1}$,  where  $n=h(h-1)q+1$ with positive integer $q$, then $\delta(G)=h-1$, $F\nsubseteq G$, $G\nsubseteq S_{n,h}$,
and $G\ncong L_{\frac{n-1}{h},h}$. Thus Theorem~\ref{Thm4} does not hold for $\delta(G)=h-1$. For another example,  let
 $F=P_4\bigcup P_2\bigcup P_3$ and $h=3$. If  $G=L_{\frac{n-1}{2},2}$, where $n=6q+1$ with positive integer $q$, then
 $\delta(G)=2$, $F\nsubseteq G$, and $G\nsubseteq S_{n,3}$. %and $G\ncong L_{\frac{n-1}{3},3}$.
 Thus Theorem~\ref{Thm4} does not hold for $\delta(G)=h-1$. Therefore the condition $\delta(G)\geq h$  in Theorem~\ref{Thm4}  is  necessary and best possible.

%By similar argument above, the condition $\delta(G)\geq h$  in Theorem~\ref{Thm4}  is also best possible for $F=P_{2b_1+2}\bigcup P_{2b_1+1}$.

\begin{Theorem}\label{Thm5}
Let $F=(\bigcup_{i=1}^k P_{2a_i})\bigcup (\bigcup_{i=1}^2 P_{2b_i+1})$, $ h=\sum _{i=1}^{k}a_i+\sum _{i=1}^2 b_i-1\geq2$,
and $G$ be a 2-connected graph of order $n$,
where $k\geq0$, $a_1\geq \dots \geq a_{k}\geq1$, $b_1\geq b_2$, and $n\geq4(2h+1)^2\binom{2h+1}{h}$.

(a). If  $\delta(G)\geq h$ and $k=0$, then $F\subseteq G$, unless one of the following holds:\\
 (i). $G\subseteq S^+_{n,h}$;\\
  (ii). $F=P_7\bigcup P_3$  and $G\subseteq K_2\bigvee \frac{n-2}{2}K_2$, where $n$ is even; \\
   (iii).    $F=P_9\bigcup P_3$ and $G\subseteq K_3\bigvee \frac{n-3}{2}K_2$, where $n$ is odd.

(b). If $\delta(G)\geq h$ and  $k\geq1$, then $F\subseteq G$, unless one of the following holds:\\
(iv). $G\subseteq S_{n,h}$;\\
  (v). $F= P_4\bigcup 2P_3$ and $G\subseteq K_2\bigvee \frac{n-2}{2}K_2$, where $n$ is even;\\
   (vi). $F=P_6\bigcup 2P_3$ and $G\subseteq K_3\bigvee \frac{n-3}{2}K_2$, where  $n$ is odd.
\end{Theorem}

Let $H_n^1$ be a graph of order $n$, where $n\geq 7$, obtained from $S_{n,2}$ and $K_3$ by identifying a vertex with maximum degree of $S_{n,2}$ with
a vertex of $K_3$ (see Fig.3).
Let $H_n^2$ be a graph order $n$, where $n\geq 9$, obtained from $H_{n-2}^1$ and $K_3$ by identifying a vertex with the second largest degree of  $H_{n-2}^1$ with
a vertex of $K_3$ (see Fig.3).

Let $U_{3,h}$ be a graph of order $3h+3$ obtained from $3K_{h+1}$ and $K_3$ by identifying  every vertex of  $K_3$ with
a vertex of $K_{h+1}$, respectively (see Fig.3).

\vspace{4mm}
\begin{center}
\begin{tikzpicture}[scale=1.4]
%\path (0,0)  coordinate (P1);\path (1.5,0)  coordinate (P2);
%\path (0.75,0.4)  coordinate (P3);\path (0.75,0.75)  coordinate (P4);\path (0.75,-0.75)  coordinate (P5);
%\path (0.75,-0.3)  coordinate (Q1);\path (0.75,-0.4)  coordinate (Q2);\path (0.75,-0.5)  coordinate (Q3);
%\foreach \i in {1,2,3,4,5}
%{\fill (P\i) circle (1pt);}
%\foreach \i in {1,2,3}
%{\fill (Q\i) circle (0.7pt);}
%\foreach \i in {1,2}
%{
%\foreach \j in {3,4,5}
%{
%\draw (P\i) -- (P\j);
%}
%}
%\draw (P1) -- (P2);
%\path (0.8,-1.3) node () {$K_{2,n-2}+e$};

\path (2.5,0)  coordinate (P1);\path (4,0)  coordinate (P2);
\path (3.25,0.4)  coordinate (P3);\path (3.25,0.75)  coordinate (P4);\path (3.25,-0.75)  coordinate (P5);
\path (3.25,-0.3)  coordinate (Q1);\path (3.25,-0.4)  coordinate (Q2);\path (3.25,-0.5)  coordinate (Q3);
\path (4.5,0.5)  coordinate (R1);\path (4.5,-0.5)  coordinate (R2);
\foreach \i in {1,2,3,4,5}
{\fill (P\i) circle (1pt);}
\foreach \i in {1,2,3}
{\fill (Q\i) circle (0.7pt);}
\fill (R1) circle (1pt);\fill (R2) circle (1pt);
\foreach \i in {1,2}
{
\foreach \j in {3,4,5}
{
\draw (P\i) -- (P\j);
}
}
\draw (P1) -- (P2);
\draw (P2) -- (R1)--(R2)--cycle;
\path (3.3,-1.3) node () {$H_n^1$};

\path (6,0)  coordinate (P1);\path (7.5,0)  coordinate (P2);
\path (6.75,0.4)  coordinate (P3);\path (6.75,0.75)  coordinate (P4);\path (6.75,-0.75)  coordinate (P5);
\path (6.75,-0.3)  coordinate (Q1);\path (6.75,-0.4)  coordinate (Q2);\path (6.75,-0.5)  coordinate (Q3);
\path (8,0.5)  coordinate (R1);\path (8,-0.5)  coordinate (R2);
\path (5.5,0.5)  coordinate (R3);\path (5.5,-0.5)  coordinate (R4);
\foreach \i in {1,2,3,4,5}
{\fill (P\i) circle (1pt);}
\foreach \i in {1,2,3}
{\fill (Q\i) circle (0.7pt);}
\foreach \i in {1,2,3,4}
{\fill (R\i) circle (1pt);}
\foreach \i in {1,2}
{
\foreach \j in {3,4,5}
{
\draw (P\i) -- (P\j);
}
}
\draw (P1) -- (P2);
\draw (P2) -- (R1)--(R2)--cycle;
\draw (P1) -- (R3)--(R4)--cycle;
\path (6.8,-1.3) node () {$H_n^2$};

\draw (9.1,-0.5) circle (0.3cm);
\draw (11.1,-0.5) circle (0.3cm);\draw (10.1,0.5) circle (0.3cm);
\path (9.4,-0.6)  coordinate (P1);\path (10.8,-0.6)  coordinate (P2); \path (10.1,0.2)  coordinate (P3);
\fill (P1) circle (1pt); \fill (P2) circle (1pt);\fill (P3) circle (1pt);
\draw (P1)--(P2)--(P3)--cycle;
\path (9.1,-0.5) node () {\scriptsize{$K_{h+1}$}};
\path (10.1,0.5) node () {\scriptsize{$K_{h+1}$}};\path (11.1,-0.5) node () {\scriptsize{$K_{h+1}$}};
\path (10.2,-1.3) node () {$U_{3,h}$};

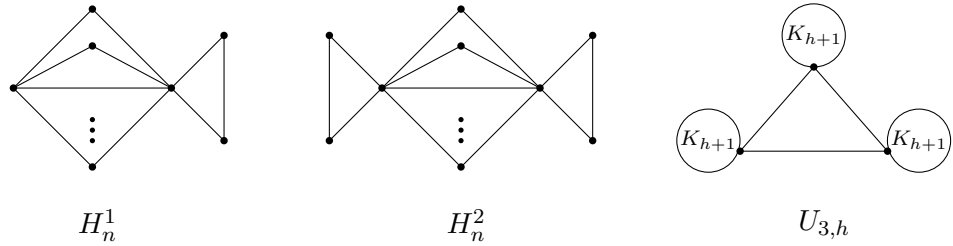
\path (6.5,-2) node () {Fig.3. Graphs  $H_n^1$,  $H_n^2$, and $U_{3,h}$.};
\end{tikzpicture}
\end{center}

For $h\geq 2$, let $F_{t_1,t_2,h,h+1}$ be a graph  of order $t_1h+(t_2+1)(h+1)+1$
obtained from $L_{t_1,t_2,h,h+1}$ and  $K_{h+1}$ by adding an edge joining the center of $L_{t_1,t_2,h,h+1}$ and a vertex of $K_{h+1}$ (see Fig.4).
Moreover,  for $h\geq 2$, let $T_{t_1,t_2,h,h+1}$ be a graph of order $t_1h+(t_2+2)(h+1)+1$
obtained from $L_{t_1,t_2,h,h+1}$ and $2K_{h+1}$ by adding an edge joining the center of $L_{t_1,t_2,h,h+1}$ and a vertex of each $K_{h+1}$, respectively (see Fig.4).

\vspace{4mm}
\begin{center}
\begin{tikzpicture}[scale=1.4]
\path (0.7,0)  coordinate (R1);\path (0,-0.6)  coordinate (R2);\path (0,0.6)  coordinate (R3);
\path (2,0)  coordinate (P1);\path (1,0.3)  coordinate (P2);\path (1.5,1)  coordinate (P3);
\path (2.5,1)  coordinate (P4);\path (3,0.3)  coordinate (P5);
\path (1,-0.3)  coordinate (P6);\path (1.5,-1)  coordinate (P7);
\path (2.5,-1)  coordinate (P8);\path (3,-0.3)  coordinate (P9);
\path (1.8,0.7)  coordinate (Q1);\path (2,0.7)  coordinate (Q2);\path (2.2,0.7)  coordinate (Q3);
\path (1.8,-0.7)  coordinate (Q4);\path (2,-0.7)  coordinate (Q5);\path (2.2,-0.7)  coordinate (Q6);
\foreach \i in {1,...,9}
{\fill (P\i) circle (1pt);}
\foreach \i in {1,...,3}
{\fill (R\i) circle (1pt);}
\foreach \i in {1,...,6}
{\fill (Q\i) circle (0.5pt);}
\draw (P1)--(P2)--(P3)--cycle;
\draw (P1)--(P4)--(P5)--cycle;
\draw (P1)--(P6)--(P7)--cycle;
\draw (P1)--(P8)--(P9)--cycle;
\draw (R1)--(R2)--(R3)--cycle;
\draw (P1)--(R1);
\path (1.5,0.4) node () {\scriptsize{$K_{h+1}$}};\path (2.6,0.4) node () {\scriptsize{$K_{h+1}$}};
\path (1.5,-0.4) node () {\scriptsize{$K_{h+2}$}};\path (2.6,-0.4) node () {\scriptsize{$K_{h+2}$}};\path (0.3,0) node () {\scriptsize{$K_{h+1}$}};
\path (2,0.9) node () {$t_1$};\path (2,-0.9) node () {$t_2$};
\path (2,-1.5) node () {$F_{t_1,t_2,h,h+1}$};

\path (4.7,0)  coordinate (R1);\path (4,-0.6)  coordinate (R2);\path (4,0.6)  coordinate (R3);
\path (7.3,0)  coordinate (R4);\path (8,-0.6)  coordinate (R5);\path (8,0.6)  coordinate (R6);
\path (6,0)  coordinate (P1);\path (5,0.3)  coordinate (P2);\path (5.5,1)  coordinate (P3);
\path (6.5,1)  coordinate (P4);\path (7,0.3)  coordinate (P5);
\path (5,-0.3)  coordinate (P6);\path (5.5,-1)  coordinate (P7);
\path (6.5,-1)  coordinate (P8);\path (7,-0.3)  coordinate (P9);
\path (5.8,0.7)  coordinate (Q1);\path (6,0.7)  coordinate (Q2);\path (6.2,0.7)  coordinate (Q3);
\path (5.8,-0.7)  coordinate (Q4);\path (6,-0.7)  coordinate (Q5);\path (6.2,-0.7)  coordinate (Q6);
\foreach \i in {1,...,9}
{\fill (P\i) circle (1pt);}
\foreach \i in {1,...,6}
{\fill (R\i) circle (1pt);}
\foreach \i in {1,...,6}
{\fill (Q\i) circle (0.5pt);}
\draw (P1)--(P2)--(P3)--cycle;
\draw (P1)--(P4)--(P5)--cycle;
\draw (P1)--(P6)--(P7)--cycle;
\draw (P1)--(P8)--(P9)--cycle;
\draw (R1)--(R2)--(R3)--cycle;
\draw (R4)--(R5)--(R6)--cycle;
\draw (P1)--(R1);\draw (P1)--(R4);
\path (5.5,0.4) node () {\scriptsize{$K_{h+1}$}};\path (6.6,0.4) node () {\scriptsize{$K_{h+1}$}};
\path (5.5,-0.4) node () {\scriptsize{$K_{h+2}$}};\path (6.6,-0.4) node () {\scriptsize{$K_{h+2}$}};
\path (4.3,0) node () {\scriptsize{$K_{h+1}$}};\path (7.7,0) node () {\scriptsize{$K_{h+1}$}};
\path (6,0.9) node () {$t_1$};\path (6,-0.9) node () {$t_2$};
\path (6,-1.5) node () {$T_{t_1,t_2,h,h+1}$};
\path (4,-2.3) node () {Fig.4. $F_{t_1,t_2,h,h+1}$  and $T_{t_1,t_2,h,h+1}$.};
\end{tikzpicture}
\end{center}

\vspace{4mm}
% Let
%$K_{2,n-2}+e$ be  the graph  of order $n$ obtained from $K_{2,n-2}$ by adding an edge to two vertices with the maximum degree for $n\ge 4$  (see Fig.3).

\begin{Theorem}\label{Thm2}
Let $F=(\bigcup_{i=1}^k P_{2a_i})\bigcup (\bigcup_{ i=1}^2 P_{2b_i+1})$, $ h=\sum_{i=1}^{k}a_i+\sum _{i=1}^2 b_i-1$,
and  $G$ be a connected  graph of order $n$ with at least one cut vertex,  where    $k\geq0$,
$a_1\geq \dots \geq a_{k}\geq1$, $b_1\geq b_2\geq1$,  and $n\geq2h+4$.

(a).  If   $\delta(G)\geq h\geq1$ and  $k=0$, then $F\subseteq G$, unless one of the following holds:\\
(i)  $F=P_5\bigcup P_3$ and either $G\subseteq H_n^1$ or $G\subseteq H_n^2$;\\
 (ii) $F=P_{2b_1+1}\bigcup P_{2b_1-1}$ and $G=L_{t,h}$, where $n=th+1$;\\
  (iii) $F=2P_{2b_1+1}$ and $G=U_{3,h}$, where $n=3h+3$;\\
 (iv) $F=2P_{2b_1+1}$ and $G\subseteq L_{t_1,t_2,h,h+1}$, where   $n=t_1h+t_2(h+1)+1$\\
 (v) $F=2P_{2b_1+1}$ and  $ G\subseteq F_{t_1,t_2,h,h+1}$, where  $n=t_1h+(t_2+1)(h+1)+1$\\
 (\romannumeral6)   $F=2P_{2b_1+1}$ and $G\subseteq T_{t_1,t_2,h,h+1}$, where  $n=t_1h+(t_2+2)(h+1)+1$.

 (b). If $\delta(G)\geq h\geq2$ an $k\geq1$, then $F\subseteq G$, unless  $F=P_2\bigcup 2P_{2b_1+1}$ and $G=L_{t,h}$, where  $n=th+1$.
\end{Theorem}

%Apply similar argument to Theorem~\ref{Thm1} and Theorem~\ref{Thm4}, the condition $\delta(G)\geq h$ in Theorem~\ref{Thm2} is best possible.

%\vspace{4mm}
%\noindent\textbf{Organization:}
The rest of this paper is organized as follows. In Section 2, some technical lemmas are provided.
 In Section 3, we present the proofs of Theorems~\ref{Thm1} and \ref{Thm4}. In Section~4, we  present the proofs of Theorems~\ref{Thm5} and \ref{Thm2}.

\section{Preliminary}

In order to prove Theorems~\ref{Thm1}--\ref{Thm2}, we require several known and new technical lemmas.
%\begin{Lemma}\label{Lem11}\cite{EG}
%Let $G$ be a  graph of order $n$ and $h\geq2$. If $e(G)>\frac{hn}{2}$, then $G$ contains a path $P_{h+2}$.
%\end{Lemma}

\begin{Lemma}\label{Lem7}\cite{EG}
Let  $G$ be a 2-connected graph of order $n$ with  $u_1\in V(G)$. If $d(u)\geq h\geq2$ for any vertex $u$ different from  $u_1$, then there is a path $P_{\min\{n, 2h\}}$ with end vertex  $u_1$ and $P_{\min\{n, 2h\}}\subseteq G$.
%contains a path of
%order $\min\{n, 2h\}$ with an end vertex $u_1$ as a subgraph.
\end{Lemma}

\begin{Lemma}\label{Lem9}\cite{Dirac}
Let $G$ be a 2-connected graph of order  at least  $2h$, where $h\geq2$. If $\delta(G)\geq h$, then $C_l\subseteq G$, where $l\geq 2h$.
%contains a cycle of order at least $2h$ as a subgraph.
\end{Lemma}

\begin{Lemma}\label{Lem5}\cite{LLP}
Let  $G$ be a graph of order $n$ with a set $P$ of $p$ vertices, where $p\geq2$ and  $n\geq4p^2\binom{p}{\lfloor\frac{p}{2}\rfloor}$. If
$e(P,V(G)\backslash P)\geq (\lfloor\frac{p}{2}\rfloor-\frac{1}{2})n$, then $P$ contains a subset of $\lfloor\frac{p}{2}\rfloor$ vertices
with a common neighborhood of $p$ vertices in $V(G)\backslash P$.
\end{Lemma}

\begin{Lemma}\label{Lemma18}
Let $G$ be a connected graph of order $n$ with  a longest cycle $C_l$ and $\delta(G)\geq h\geq2$.
Let $U=V(G)\backslash V(C_l)$. \\
(i). If $l=2h$ and $U$ is an independent set, then $G\subseteq S_{n,h}$.\\
(ii).  If $l=2h+1$ and $P_{2h+3}\nsubseteq G$, then $G\subseteq S^+_{n,h}$.\\
(iii).   If $l=2h+2$ and $P_{2h+4}\nsubseteq G$, then  $N_{C_l}(u_1)=N_{C_l}(u_2)$ and $h\leq d_{C_l}(u_1)=d_{C_l}(u_2)\leq h+1$ for every pair of vertices $u_1,u_2\in U$.
\end{Lemma}
\begin{Proof}
Let $C_l=v_1v_2\cdots v_lv_1$. Since $C_l$ is a longest cycle, none of vertices in $U$ is adjacent to any  two consecutive  vertices of $C_l$.

$(i)$. Since $l=2h$ and $U$ is an independent set, $d_{C_l}(u)=h$ for all $u\in U$. Furthermore either $N_{C_l}(u)=\{v_1,v_3,\ldots,v_{2h-1}\}$
or $N_{C_l}(u)=\{v_2,v_4,\ldots,v_{2h}\}$ for all $u\in U$. In fact, if there exist two distinct vertices $u_1,u_2\in U$ such that
 $N_{C_l}(u_1)=\{v_1,v_3,\ldots,v_{2h-1}\}$ and $N_{C_l}(u_2)=\{v_2,v_4,\ldots,v_{2h}\}$,
 then a cycle $u_1v_1v_{2h}u_2 v_2 v_3v_4 \cdots v_{2h-3}v_{2h-2}v_{2h-1}u_1$ is a longer cycle than $C_l$,  a contradiction.
 Hence,   we assume  without loss of generality that $N_{C_l}(u)=\{v_2,v_4,\ldots,v_{2h}\}$ for all $u\in U$.
Moreover, $\{v_1,v_3,\ldots,v_{2h-1}\}$ is an independent set. In fact, if there exists an edge $v_{2s+1}v_{2t+1}\in E(G)$
with $1\leq 2s+1<2t+1\leq 2h-1$,  then a cycle $uv_{2t}v_{2t-1}v_{2t-2}\cdots v_{2s+2}v_{2s+1}v_{2t+1}v_{2t+2}\cdots  v_{2h-1}v_{2h}v_1v_2\cdots v_{2s-1}v_{2s}u$
for $u\in U$ is a longer cycle than $C_l$, a contradiction.
Hence $G\subseteq S_{n,h}$.

 $(ii)$.  If $l=2h+1$ and $P_{2h+3}\nsubseteq G$, then $U$ is an independent set and  $d_{C_l}(u)=h$ for all $u\in U$.
Furthermore, we claim that $N_{C_l}(u_1)=N_{C_l}(u_2)$ for any two vertices $u_1, u_2\in U$. In fact,  we suppose without loss of generality that  there exists
 $v\in N_{C_l}(u_2)\backslash N_{C_l}(u_1)$.
Since  $P_{2h+3}\nsubseteq G$, neither of two neighbors of $v$ along $C_l$  belongs to $N_{C_l}(u_1)$.
Hence  $N_{C_l}(u_1)$ is a subset of a set consisting of $2h-2$ consecutive vertices of ${C_l}$. Furthermore,  since ${C_l}$ is a longest cycle, $u_1$ is not adjacent to any two
consecutive vertices of ${C_l}$. Hence we $d_{C_l}(u_1)\le h-1$, a contradiction.
We assume without loss of generality that $N_{C_l}(u)=\{v_2,v_4,\ldots,v_{2h}\}$ for all $u\in U$.  Moreover,   $\{v_1,v_3,\ldots, v_{2h-1}\}$
is an independent set. Otherwise there exists an edge $v_{2s+1}v_{2t+1}\in E(G)$ with $1\leq 2s+1<2t+1\leq 2h-1$,
 which causes a longer cycle $uv_{2t}v_{2t-1}v_{2t-2} \cdots v_{2s+2}v_{2s+1}v_{2t+1}v_{2t+2}\cdots v_{2h} v_{2h+1}v_1\cdots v_{2s}u$ for $u\in U$  than $C_l$,
 a contradiction.
Similarly,  $\{v_3,\ldots, v_{2h+1}\}$ is also an independent set.
Hence  $G\big[U\bigcup \{v_1,\ldots, v_{2h+1}\}\big]$  contains at most  one edge $v_1v_{2h+1}$. Therefore $G\subseteq S^+_{n,h}$.

$(iii).$    If $l=2h+2$ and $P_{2h+4}\nsubseteq G$, then  $U$ is an independent set and   $h\leq d_{C_l}(u)\leq h+1$ for all $u\in U$.
Furthermore,  $N_{C_l}(u_1)=N_{C_l}(u_2)$ for any  two vertices $u_1, u_2\in U$. In fact,  suppose without loss of generality that there exists a vertex  $v\in N_{C_l}(u_2)\backslash N_{C_l}(u_1)$. Since  $P_{2h+4}\nsubseteq G$, neither of two neighbors of $v$ along $C_l$  belongs to $N_{C_l}(u_1)$.
 Hence  $N_{C_l}(u_1)$ is a subset of a set consisting of $2h-1$ consecutive vertices of ${C_l}$. Furthermore,  since ${C_l}$ is a longest cycle, $u_1$ is not adjacent  to any two consecutive vertices of $C_l$. Hence $d_{C_l}(u_1)\le h-1$, a contradiction.  Hence $N_{C_l}(u_1)=N_{C_l}(u_2)$ and the assertion holds.
\end{Proof}

\begin{Lemma}\label{Lemma19}
Let $G$ be a connected graph of order $n$ with  a longest cycle $C_l$ and minimum degree $\delta(G)$, where  $l\leq 2h+1$  and $\delta(G)\geq h\geq2$.
Let  $U=V(G)\backslash V(C_l)$.\\
(i).  If $G[U]$ is $P_3$-free, then  $N_{C_l}(u_1)=N_{C_l}(u_2)$ for every edge $u_1u_2\in E(U)$.\\
(ii). If $G[U]$ is $P_4$-free, then $N_{C_l}(u_1)=N_{C_l}(u_3)$ for every  $P_3=u_1u_2u_3\subseteq G[U]$.
\end{Lemma}

\begin{Proof}
$(i)$. Suppose  that there exists  a vertex $v\in N_{C_l}(u_2)\backslash N_{C_l}(u_1)$  for some edge $u_1u_2\in E(U)$.
 Since $C_l$ is a  longest cycle, the distance along $C_l$ between $v$ and any vertex in $N_{C_l}(u_1)$ is at least $3$.  Thus $N_{C_l}(u_1)$ is a subset of  a set consisting of
 at most $2h-5$ consecutive  vertices of ${C_l}$. Furthermore,  since $C_l$ is a  longest cycle, $u_1$ is not adjacent to any two consecutive vertices of $C_l$.
These discussions imply that $d_{C_l}(u_1)\le h-2$.  However, since $P_3\nsubseteq G[U]$, we have $d_{C_l}(u_1)\geq d_G(u_1)-1\geq h-1$,
  a contradiction. Hence the assertion holds.%Hence $N_{C_l}(u_2)\subseteq N_{C_l}(u_1)$.  Similarly, $N_{C_l}(u_1) \subseteq N_{C_l}(u_2)$. Therefore $N_{C_l}(u_1)= N_{C_l}(u_2)$ and .

$(ii)$. Suppose that there exists a vertex $v\in N_{C_l}(u_3)\backslash N_{C_l}(u_1)$ for some path $P_3=u_1u_2u_3\subseteq G[U]$.
Since $C_l$ is a  longest cycle,  the distance along $C_l$ between $v$ and any vertex in $N_{C_l}(u_1)$ is at least $4$. Hence $N_{C_l}(u_1)$ is a subset of a set consisting of
 at most $2h-7$ consecutive vertices of $C_l$. Furthermore,  since $C_l$ is a  longest cycle,  $u_1$ is not adjacent to any  two consecutive vertices of $C_l$.
These discussions imply that $d_{C_l}(u_1)\leq h-3$. However, since $P_4\nsubseteq G[U]$, we have $d_{C_l}(u_1)\geq d_G(u_1)-2\geq h-2$,  a contradiction.
 %Hence $N_{C_l}(u_3)\subseteq N_{C_l}(u_1)$. Similarly, $N_{C_l}(u_1) \subseteq N_{C_l}(u_3)$.
% Therefore $N_{C_l}(u_1)= N_{C_l}(u_3)$ and
Hence the assertion holds.
\end{Proof}

\begin{Lemma}\label{Lemma20}
Let $F=(\bigcup_{i=1}^k P_{2a_i})\bigcup (\bigcup_{i=1}^2 P_{2b_i+1})$ and $ h=\sum _{i=1}^{k}a_i+\sum _{i=1}^2 b_i-1\geq2$, where $k\geq0$,
 $a_1\geq \cdots \geq a_{k}\geq1$,  and $b_1\geq b_2\geq1$.
 Let $H$, i.e., $H=(X,Y;E)$, be a complete bipartite graph with partition size $|X|=h$ and $|Y|=h+2$.
\\
$(i).$ If  $k\geq1$ and  the graph $G$  is  obtained from $H$ and $P_3$ by identifying  a vertex $u\in X$ with an end vertex of $P_3$, then  $F\subseteq G$.\\
$(ii).$  If   the graph $G$  is  obtained from $H$ and  $P_4$ by identifying a vertex $u\in X$ with an end vertex of $P_4$, then  $F\subseteq G$.\\
$(iii).$  If  the graph $G$  is obtained from $H-v$ with $v\in X$ and $P_6$ by identifying $u\in X$ with an end vertex of $P_6$, then  $F\subseteq G$.\\
\end{Lemma}

\begin{Proof}
$(i).$ Since $H$ is a complete bipartite graph with partition size $|X|=h$ and $|Y|=h+2$,  there exist three disjoint complete bipartite subgraphs $H_i$ of $H$, i.e.,  $H_i=(X_i, Y_i: E_i)$ for $1\leq i\leq 3$, where
 $|X_1|=|Y_1|=\sum _{i=1}^k a_i-1$  with $u\in X_1$, $|X_2|=b_1, |Y_2|=b_1+1$,  $|X_3|=b_2$, and  $|Y_3|=b_2+1$.
Hence there exists a path $P_{\sum _{i=1}^k 2a_i-2}$%$Q_1$ of order  $\sum _{i=1}^k 2a_i-2$
  in $H_1$ with one end vertex $u$.  In addition, there exist two disjoint paths $P_{2b_1+1}$ in $H_2$ and $P_{2b_2+1}$ in $H_3$.
%$Q_2$ of order $2b_1+1$ in $H_2$ and $Q_3$ of order $2b_2+1$ in $H_3$.
Hence $F\subseteq G$. So $(i)$ holds.
By a similar argument in $(i),$ it is easy to see that $ (ii)$ and $(iii)$ hold.
%?????We just prove $(i)$ since $(ii)$ and $(iii)$ follow by similar arguments as the proof of $(i)$.\\
\end{Proof}

\begin{Lemma}\label{Lemma16}
Let $H$ be a 2-connected graph of order  $n$,  where $n\geq2h+1$ and  $2\leq h\leq3$, and
 $u_1\in V(H)$ such that  $d_H(u)\geq h$ for every $u$ different from $u_1$.
 Let $G$ be a graph obtained from $H$ and $P_t$ by identifying $u_1$ with an end vertex of $P_t$, where $3\leq t\leq4$.\\
$(i).$ If $h=2$ and  $t=3$, then $2P_2 \bigcup P_3\subseteq G$, $P_4 \bigcup P_3\subseteq G$, and $P_2 \bigcup P_5\subseteq G$.
$ (ii).$ If $h=2$ and  $t=4$, then $P_5 \bigcup P_3\subseteq G$.\\
$ (iii).$  If $h=3$ and  $t=4$, then $P_7 \bigcup P_3\subseteq G$ and $2P_5\subseteq G$.\\
\end{Lemma}

\begin{Proof}
($i$).  Let $C_l$ be a longest cycle  in $H$.  Since $H$ be a 2-connected graph,  $\delta(H)\geq2$.
By Lemma~\ref{Lem9},  $l\geq 4$. Since $H$ is connected, if $l\geq5$ then  $2P_2 \bigcup P_3\subseteq G$, $P_4 \bigcup P_3\subseteq G$,
and $P_2 \bigcup P_5\subseteq G$. Since $n\geq 5$, if $l=4$  then  there exists a vertex $v\in V(H)\backslash V(C_l)$
adjacent to some vertex of $C_l$.
We have $2P_2 \bigcup P_3\subseteq G$, $P_4 \bigcup P_3\subseteq G$, and $P_2 \bigcup P_5\subseteq G$. Thus the assertion holds.

($ii$). The proof of $(ii)$ is similar to that of $(i)$ and the detail is omitted.
%Since $H$ is a 2-connected graph of order  at least 5, $\delta(G)\geq2$.
%Let $C$ be the longest cycle in $H$ of order $l$.
%By Lemma~\ref{Lem9}, $l\geq 4$. If $l\geq5$, then   $P_5 \bigcup P_3\subseteq G$.
%If $l=4$ with $u_1\notin V(C)$, then $P_5 \bigcup P_3\subseteq G$ by $n\geq5$ and  $H$ being connected.
%If $l=4$  with $v_1\in V(C)$,  then  there exists a vertex in $H$ outside $C$ adjacent to some vertex in $C$ by $n\geq5$ and $H$ being connected.
%Thus $P_5 \bigcup P_3\subseteq G$.

%By Lemma~\ref{Lem7}, $H$ has a path $P$ of order $2h$ with end vertex $u_1$. Denote this path by
%$P=u_1u_2\ldots u_{2h}$. Since $\large|V(H)\large | \geq 2h+1$ and $H$ is connected,
%there exists a vertex $u\in V(H)\backslash V(P)$adjacentto some vertex in $P$.

% $(i)$  If  $\{u_1,u_4\}\bigcap N_H(u)\neq \varnothing$ or $\{u_2,u_3\}\subset N_H(u)$,
%then $2P_2 \bigcup P_3\subset G$, $P_4 \bigcup P_3\subset G$ and $P_2 \bigcup P_5\subset G$.
%Thus we can suppose that $\{u_1,u_4\}\bigcap N_H(u)=\varnothing$ and $\{u_2,u_3\}\nsubseteq N_H(u)$.
%Since $H$ is connected, then $u$ is adjacent to only one of $\{u_2,u_3\}$.
%Since $d_H(u)\geq 3$, there exists a vertex $v\in V(H)\backslash V(P)$adjacentto $u$ and
%thus  $2P_2 \bigcup P_3\subset G$, $P_4 \bigcup P_3\subset G$ and $P_2 \bigcup P_5\subset G$.
%
%  $(ii)$ If $\{u_1,u_3, u_4\}\bigcap N_H(u)\neq \varnothing$, then $P_5\bigcup P_3\subset G$.
%  Thus we can suppose that $N_H(u)\bigcap V(P)=\{v_2\}$. Since $d_H(u)\geq 2$,
%  there exists a vertex $v\in V(H)\backslash V(P)$adjacentto $u$ and thus
%  $P_5 \bigcup P_3\subset G$.

$(iii)$. By Lemma~\ref{Lem7}, $H$ has a path $P_6$ with an end vertex $u_1$. Let $P_6=
u_1u_2\cdots u_6$. Since $n\geq 7$,
there exists a vertex $v\in V(H)\backslash V(P_6)$ adjacent to some vertex of $P_6$.
%Let $N_P(v)=N_H(v)\bigcap V(P)$.
If $u_1$, $u_3$, or $u_6\in N_{P_6}(v)$,  then  $P_7 \bigcup P_3\subseteq G$ and $2 P_5\subseteq G$.
Furthermore, if $\{u_4,u_5\}\subseteq N_{P_6}(v)$, then $P_7 \bigcup P_3\subseteq G$ and $2 P_5\subseteq G$.
Next  we assume  $N_{P_6}(v)\in\{\{v_2,v_4\}, \{v_2,v_5\},\{u_2\},
\{u_4\},\{u_5\}\}$.
Since $d_H(v)\geq3$, if $N_{P_6}(v)\in\{\{v_2,v_4\}, \{v_2,v_5\}\}$ then  there exists a vertex in $V(H)\backslash V({P_6})$ adjacent to $v$.
Thus $P_7 \bigcup P_3\subseteq G$ and $2 P_5\subseteq G$.
Moreover, since $d_H(v)\geq3$, if  $N_{P_6}(v)\in\{\{u_4\},\{u_5\}\}$ then there exists two distinct vertices in $V(H)\backslash V({P_6})$ adjacent to $v$. Thus $P_7 \bigcup P_3\subseteq G$ and $2 P_5\subseteq G$.
Finally, since $H$ is 2-connected, if  $N_{P_6}(v)=\{u_2\}$  then there exists a vertex in $V(H)\backslash V({P_6})$ adjacent to some vertex in $V({P_6})\backslash \{u_2\}$ ($\{u_2\}$ can not separate $V(P)\backslash \{u_2\}$ from the rest).
By repeating the arguments above,  $P_7 \bigcup P_3\subseteq G$ and $2 P_5\subseteq G$. Thus the assertion holds.
\end{Proof}

\begin{Lemma}\label{Lemma17}
Let $H$ be a 2-connected graph of order $n$ and $H\nsubseteq S_{n,2}$,  where $n\geq6$.  If a graph $G$  is  obtained from $H$ and $P_3$ by
identifying a vertex $u_1\in V(H)$ with an end vertex of  $P_3$,  then $P_5\bigcup P_3\subseteq G$.
\end{Lemma}

\begin{Proof}
Suppose  $P_5\bigcup P_3\nsubseteq G$.
Let $C_l$ be a longest cycle  in $H$.   Since $H$ is 2-connected, ¡¡$\delta(H)\geq2$.
By Lemma~\ref{Lem9}, $l\geq 4$.

\textbf{Case 1}: $l\geq 6$, or $l=5$  with $u_1\notin V(C_l)$.
According to the structure of $G$, we have $P_5\bigcup P_3\subseteq G$, a contradiction.

\textbf{Case 2}: $l=5$  with $u_1\in V(C_l)$.
Since  $n\geq 6$, there exists a vertex $u\in V(H)\backslash V(C_l)$ adjacent to some vertex of $C_l$.
Furthermore,  since $P_5\bigcup P_3\nsubseteq G$,  we have $N_{C_l}(u)=\{u_1\}$.
Moreover,  since $\delta(H)\geq2$, there exists a vertex in $V(H)\backslash V(C_l)$  adjacent to $u$. Hence $P_5\bigcup P_3\subseteq G$, a contradiction.

 {\bf Case 3:} $l=4$.  There exists a shortest path $P_k$  between $u_1$ and  some vertex of $C_l$
such that  $|V(P_k)\bigcap V(C_l)|=1$.
 %Let  $k$ be the order of  the shortest path $P$  between $u_1$ and  some vertex of $C_l$
%such that  $|V(P)\bigcap V(C_l)|=1$.
 Since $P_5\bigcup P_3\nsubseteq G$, we have $1\leq k\leq 2$.

 {\bf Subcase 3.1:} $k=2$. Let $C_4=v_1v_2v_3 v_4v_1$, $V(C_4)\bigcap V(P_2)=\{v_1\}$,  and $U=V(H)\backslash (V(C_4)\bigcup\{u_1\})$.
 It is easy to see  that $U\neq \varnothing$.
First we assume  that  $U$ is an independent set. Since  $C_4$ is a  longest cycle in $H$ and $\delta(H)\geq2$,
$N_H(u)\in\{\{u_1,v_1\},\{v_1,v_3\},\{v_2,v_4\}\}$ for every $u\in U$. However, in each case, $P_5\bigcup P_3\subseteq G$,
 a contradiction.
 Next we assume   that  $H[U]$ contains at least one edge. Then there exists one edge  $w_1w_2$ such that one of $w_1$ and $w_2$ is adjacent to some vertex in $V(C_4)\bigcup \{u_1\}$
 and therefore $P_5\bigcup P_3\subseteq G$, a contradiction.

 {\bf Subcase 3.2}:  $k=1$. Obviously $u_1\in V(C_4)$. Let $C_4=u_1u_2u_3 u_4u_1$ and $U=V(H)\backslash V(C_4)$.  Since $n\geq6$,
 we have $|U|\geq2$. First we assume  that $U$ is  an independent set.  Since  $C_4$ is a longest cycle and $\delta(H)\geq2$, either $N_H(v)=\{u_1,u_3\}$ or $N_H(v)=\{u_2,u_4\}$ for
    $v\in U$.    We claim that   $N_H(v)=\{u_1,u_3\}$ for all $v\in U$.  Otherwise $P_5\bigcup P_3\subseteq G$, a contradiction.
  Moreover, since $P_5\bigcup P_3\nsubseteq G$, it follows that $u_2u_4\notin E(G)$.
  Hence $H\subseteq S_{n,2}$, a contradiction.  Next we assume that  $H[U]$ contains at least one edge $v_1v_2$.
Furthermore,  since $P_5\bigcup P_3\nsubseteq G$,   we have $P_3\nsubseteq H[U]$.
 By Lemma~\ref{Lemma19}~(i),  $N_{C_4}(v_1)=N_{C_4}(v_2)$.  Since $C_4$ is a longest cycle,
$d_{C_4}(v_1)=d_{C_4}(v_2)=1$.
This discussion implies that the unique common neighbor  of $v_1$ and $v_2$  in $V(C_4)$ is a cut vertex of $H$, which  contradicts that  $H$ is 2-connected.
 Thus the assertion holds.
\end{Proof}

\section{Proofs of Theorems~\ref{Thm1} and \ref{Thm4}}

Now we are ready to prove Theorem~\ref{Thm1}, i.e.,

\begin{Theorem}\label{Thm1-rep}
Let $F=\bigcup_{i=1}^k P_{2a_i}$ and $ h=\sum _{i=1}^k a_i-1\geq1$, where  $k\geq2$ and
$a_1\geq \cdots \geq a_{k}\geq1$. Let $G$ be a  connected graph of order $n$, where $n\geq2h+2$. If  $\delta(G)\geq h$,
 then $F\subseteq G$, unless one of the following holds:\\
 (i). $G\subseteq S_{n,h}$;\\
  (ii). $F=2P_{2a_1}$ and $G=L_{t,h}$, where $n=th+1$.
\end{Theorem}

%\textbf{Proof of of Theorem~\ref{Thm1}.}
\begin{Proof}
Suppose $F\nsubseteq G$.
Since $\sum_{i=1}^{k}2a_i=2h+2$,  we have $P_{2h+2}\nsubseteq G$.
We consider the following two cases.

{\bf Case 1:}     $G$ is 2-connected. It follows that  $\delta (G)\geq2$.
Let $C_l$ be a longest cycle   in $G$
and  $U=V(G)\backslash V(C_l)$. By Lemma~\ref{Lem9},  $l\geq 4$.
First we claim that $h\geq 2$. In fact,
  if  $h=1$, then $F=2P_2$ and thus $F\subseteq G$, a contradiction.
By Lemma~\ref{Lem9},  $h\geq 2$ and therefore  $l\geq2h$.    Since  $n\geq2h+2$ and $P_{2h+2}\nsubseteq G$,
 we have $l\leq 2h$  and therefore $l=2h$.
Moreover, since $P_{2h+2}\nsubseteq G$, $U$ is an independent set.
By Lemma~\ref{Lemma18} $(i)$,  $G\subseteq S_{n,h}$.

{\bf Case 2:}   $G$ has at least one cut vertex.
Since $k\ge 2$, if $h=1$ then  $F=2P_2$.  Since $F\nsubseteq G$, $G$ must be a star $K_{1, n-1}$, i.e., $G=L_{n-1,1}$.
Next we assume  that $h\ge 2$.  Since $G$ has at least one cut vertex, there exist at least two end blocks.
We claim that $|V(B)|=h+1$ for every end  block  $B$ of $G$.  In fact,   since  $\delta(G)\ge h$, it follows that $|V(B)|\geq h+1$. Furthermore,
since $P_{2h+2} \nsubseteq G$, Lemma~\ref{Lem7} implies that $B$ has order at most $h+1$ and hence  $|V(B)|=h+1$.
Let $B_i$ be an end block of $G$ with a vertex $u_i$ which is a cut vertex of $G$ for $1\leq i\leq2$. It follows that $|V(B_i)|=h+1$ for $1\leq i\leq2$.
If there exists a path $P_k$ with $k\geq 2$ starting at  $u_1$
and ending at $u_2$  such that  $V(P_k)\bigcap (V(B_1)\bigcup V(B_2))=\{u_1, u_2\}$,
then  Lemma~\ref{Lem7} implies $P_{2h+2}\subseteq G$, a contradiction.
It follows that $G=L_{t,h}$, where $n=th+1$. Moreover, since  $ n\geq 2h+2$, $G$ has at least three blocks, i.e., $t\geq3$. By Lemma~\ref{Lem7},
 $P_h\bigcup P_{2h+1}\subseteq G$.
 We claim that  $k=2$ and $a_1=a_2$.
Otherwise, since
$$h=\sum \limits_{i=1}^{k}a_i-1\geq 2 a_k$$
and
$$2h+1=\sum \limits_{i=1}^{k}2a_i-1\geq \sum \limits_{i=1}^{k-1}2a_i, $$
 we have $F\subseteq G$, a contradiction.
Therefore $F=2P_{2a_1}$ and $G=L_{t,h}$, where $n=th+1$.
%If $k\geq 3$ or $k=2$ and $a_1>a_2$, then
%$$h=\sum \limits_{i=1}^{k}a_i-1\geq 2 a_k,$$
%$$2h+1=\sum \limits_{i=1}^{k}2a_i-1\geq \sum \limits_{i=1}^{k-1}2a_i.$$
%Thus $F\subset G$.  %, a contradiction.
%The contradiction implies $G=L_{t,h}$, where $n=th+1$ and $F=2P_{2a_1}$.
%qed
\end{Proof}

  Theorem~\ref{Thm4} can be stated as follows.

\begin{Theorem}\label{Thm4-rep}
Let $F=(\bigcup_{i=1}^k P_{2a_i})\bigcup P_{2b_1+1}$ and  $ h=\sum _{i=1}^k a_i+b_1-1\geq1$, where  $k\geq1$,
$a_1\geq \cdots \geq a_{k}\geq1$, and $b_1\geq1$. Let $G$ be a connected graph of order $n$, where $n\geq2h+3$. If  $\delta(G)\geq h$,
then $F\subseteq G$, unless one of the following holds:\\
(i). $G\subseteq S_{n,h}$;\\
(ii).  $F=P_6\bigcup P_3$ and $G\subseteq K_2\bigvee \frac{n-2}{2}K_2$, where $n$ is even;\\
 (iii). $F\in \{P_{2b_1}\bigcup P_{2b_1+1}, P_{2b_1+2}\bigcup P_{2b_1+1}\}$ and $G=L_{t,h}$, where  $n=th+1$.
\end{Theorem}
\begin{Proof}
%\noindent\textbf{Proof of of Theorem~\ref{Thm4}.}
Suppose  $F\nsubseteq G$.
Since $\sum _{i=1}^{k}2a_i+2b_1+1=2h+3$,  we have $P_{2h+3}\nsubseteq G$.
 We consider the following two cases.

{\bf Case 1:} $G$ is 2-connected. It follows that  $\delta (G)\geq2$. Let $C_l$, denoted by $v_1v_2\cdots v_lv_1$,  be a longest cycle  in $G$ and
  $U=V(G)\backslash V(C_l)$.   By Lemma~\ref{Lem9},  $l\geq 4$.  We claim  that $h\ge 2$. In fact, if $h=1$, then $F=P_2\bigcup P_3$.
Since  $n\geq 5$ and $l\geq 4$,  we have $F\subseteq G$, a contradiction.
By Lemma~\ref{Lem9},  $l\geq2h$. Furthermore,  since   $n\geq 2h+3$
and $P_{2h+3}\nsubseteq G$,  we have $l\leq 2h+1$.
So $2h\le l\le 2h+1$.  We consider the following two subcases.
% Let $N_C(u)=N_G(u)\bigcap V(C)$ for each $u\in U$.

\textbf{Subcase 1.1:} $l=2h$. Since $P_{2h+3}\nsubseteq G$, we have $P_3\nsubseteq G[U]$.  If $U$ is an independent set, then Lemma~\ref{Lemma18}~$(i)$ implies  $G\subseteq S_{n,h}$.
Now  we assume that $G[U]$ consists of $p$ disjoint edges and $q$ isolated vertices, i.e., $u_1u_2, u_3u_4, \ldots, u_{2p-1}u_{2p}, w_1, \ldots, w_q$,
 where $p\ge 1$ and $q\ge 0$.  By Lemma~\ref{Lemma19}~$(i)$,  $N_{C_{2h}}(u_{2i-1})=N_{C_{2h}}(u_{2i})$ for $1\leq i\leq p$.
Note that $d_{C_{2h}}(u_j)=d_G(u_j)-1\geq h-1$ for $1\leq j\leq 2p$.
Since $C_{{2h}}$ is  a longest cycle in $G$, the distance along $C_{{2h}}$ between  two vertices in $N_{C_{2h}}(u_1)$ is at least 3, which implies $3(h-1)\leq2h$.
Thus $h\leq3$.
Furthermore,  $h=3$ (Otherwise $h=2$, which implies $d_{C_4}(u_i)=1$ and $d_G(u_i)=2$ for $1\leq i\leq 2$. Thus the common neighbor of $u_1$ and $u_2$ in $V(C_4)$
is a cut vertex, which contradicts that $G$ is 2-connected).
Now let $\mathcal{A}=\{P_6\bigcup P_3, P_4\bigcup P_5,P_2\bigcup P_7, P_4\bigcup P_2\bigcup P_3, 2P_2\bigcup P_5, 3P_2\bigcup P_3\}$.
 Since $h=3$ and $k\geq1$, we have $F\in \mathcal{A}$.
Moreover, $d_{C_6}(u_{2i-1})=d_{C_6}(u_{2i})=2$ and $d_G(u_{2i-1})=d_G(u_{2i})=3$ for $1\leq i\leq p$.
  Since $C_6$ is a longest cycle in $G$, we assume without loss of generality that $N_{C_6}(u_1)=N_{C_6}(u_2)=\{v_1,v_4\}$.
 Since $F\nsubseteq G$, we have $N_{C_6}(u_{2i-1})=N_{C_6}(u_{2i})=\{v_1, v_4\}$ for $1\leq i \leq p$.  Next we claim that $q=0$.
 In fact, if $q\ge 1$, then  $d_G(w_1)=d_{C_6}(w_1)=3$, which implies   $F\subseteq G$, a contradiction.
So $n$ must be even. Hence $G\subseteq K_2\bigvee \frac{n-2}{2}K_2$.
 %In addition, by $F\nsubseteq G$, $G[\{v_2,v_3,v_5,v_6\}]$ contains exactly two disjoint edges $v_2v_3$ and $v_5v_6$.
%By $\delta(G)\geq 3$ and $d_U(v_2)=d_U(v_3)=d_U(v_5)=d_U(v_6)=0$,  $N_G(v_2)=\{v_1,v_3,v_4\}$, $N_G(v_3)=\{v_1,v_2,v_4\}$, $N_G(v_5)=\{v_1,v_4,v_6\}$  and
%$N_G(v_6)=\{v_1,v_4,v_5\}$.
Moreover, $F=P_6\bigcup P_3 $.
 %In fact, if $F\in \mathcal{A}\backslash \{P_6\bigcup P_3\}$, it is easy to see that $F\subseteq G$, a contradiction.
 The assertion holds.
% For each  vertex $u\in U\backslash \{u_1,u_2\}$,
% $N_C(u)\bigcap \{v_2,v_3,v_5,v_6\}=\varnothing$ by  $P_9\nsubseteq G$.
% Hence $N_C(u)\subseteq \{v_1,v_4\}$ and  $G[U]$ consists of disjoint edges.  Hence for each  edge $w_1w_2$ of $G[U]$,  $N_C(w_1)=N_C(w_2)=\{v_1,v_4\}$.
%%Note that  $n$ is necessarily even and $n\geq 10$.

\textbf{Subcase 1.2:}  $l=2h+1$.
By  the proof of Lemma~\ref{Lemma18} $(ii)$, we have $N_{C_{2h+1}}(u)=\{v_2, v_4, \ldots, v_{2h}\}$  for every $u\in U$.
 Hence there exist two disjoint paths $uv_2v_3\cdots v_{2b_1+1}$  and
$wv_{2b_1+2}\cdots v_{2h+1}v_1$ for two different vertices $u, w\in  U$.  Hence $F\subseteq G$ and it is a contradiction.

{\bf Case 2:}   $G$ has at least one cut vertex.  If $h=1$, then
$F=P_2\bigcup P_3$. Since  $F\nsubseteq G$,  $G$ must be a star $K_{1, n-1}$, i.e., $G=L_{n-1, 1}$.  Now we assume that
$h\ge 2$. First we prove the following three Claims.
%If $h=1$, then $F=P_2\bigcup P_3$. If $h=2$, then $F\in\{2P_2 \bigcup P_3, P_4 \bigcup P_3,P_2 \bigcup P_5\}$.

%\vspace{4mm}

\textbf{Claim 1}: $P_{h+1}\bigcup P_{2h+1}\nsubseteq G$.

%\vspace{4mm}

Suppose  $P_{h+1}\bigcup P_{2h+1}\subseteq G$.
If $\sum_{i=1}^{k}a_i\geq b_1+1$,  then
 $$h+1=\sum \limits_{i=1}^{k}a_i+b_1\geq 2b_1+1$$
 and
  $$2h+1\geq \sum \limits_{i=1}^{k}2a_i+2b_1-1\geq \sum \limits_{i=1}^{k}2a_i.$$
Thus $F \subseteq G$, a contradiction.
If  $\sum _{i=1}^{k}a_i\leq b_1$, then
 $$h+1=\sum \limits_{i=1}^{k}a_i+b_1\geq \sum \limits_{i=1}^{k}2a_i$$
 and
  $$2h+1\geq \sum \limits_{i=1}^{k}2a_i+2b_1-1\geq 2b_1+1.$$
 Thus $F \subseteq G$, a contradiction. So Claim 1 holds.

%\vspace{4mm}

\textbf{Claim 2}:  $\large|V(B)\large|= h+1$  for every end block $B$ of $G$.

% \vspace{4mm}

 Since $\delta(G)\ge h$,   it follows that  $|V(B)|\ge h+1$ for every end block $B$ of $G$. Next we  prove that $\large|V(B)\large|\leq h+1$ for every end block $B$ of $G$.
Suppose that there exists  an end block $B_1$ of order at least $h+2$.
If there exists another end block $B_2$ such that $V(B_1)\bigcap V(B_2)=\varnothing$,  then  Lemma~\ref{Lem7} implies  $P_{2h+3}\subseteq G$,
 a contradiction.
If there exist  another two end blocks $B_2$ and $B_3$ such that $B_1, B_2,$ and  $B_3$ share a common cut vertex $u$ of $G$, then Lemma~\ref{Lem7} implies that
 there exist three paths $P_{h+2}$,
$P_{h+1}$, and $P_{h+1}$ in $B_1, B_2$ and $B_3$ with an end vertex $u$, respectively.
 Hence $P_{h+1}\bigcup P_{2h+1}\subseteq G$, a contradiction.
These discussions imply that there are exactly two blocks $B_1$ and $B_2$ sharing one common cut vertex $u$ of $G$,  where $|V(B_1)|\ge h+2$ and $|V(B_2)|=n-|V(B_1)|+1$.  By  Lemma~\ref{Lem7},
there exist two paths $P_{l_1}$  and $P_{l_2}$  with an end vertex $u$ in $B_1$ and $B_2$, respectively, where $l_1\geq \min\{|V(B_1)|, 2h\}$ and $l_2\geq \min\{|V(B_2)|, 2h\}$.
Hence $P_{2h+3}\subseteq G$, a contradiction. Therefore Claim 2 holds.

\textbf{Claim 3}: $G=L_{t,h}$,  where $n=th+1$ and $F\in \{P_{2b_1}\bigcup P_{2b_1+1}, P_{2b_1+2}\bigcup P_{2b_1+1}\}$.
%\vspace{4mm}

%Firstly we  show that $G$ has precisely one cut vertex.
Choose  two end blocks  $B_1$ and $B_2$  of $G$  such that
  $p(B_1, B_2)$  is as large as possible, where  $p(B_1, B_2)$ is
  the order of a longest path between  the unique cut vertex $u_1$ of $G$ in $V(B_1)$ and  the unique cut vertex $u_2$ of $G$ in $V(B_2)$.
By Claim~2, $ |V(B_1)|= |V(B_2)|=h+1$.
If $p(B_1, B_2)\geq 3$, then  Lemma~\ref{Lem7} implies  $P_{2h+3}\subseteq G$, a contradiction.
Since $n\geq 2h+3$, if $ p(B_1, B_2)=2$ then there exists another end block $B_3$ of $G$,
sharing a common vertex $u_1$ with $B_1$ or a common vertex $u_2$ with $B_2$. By Lemma~\ref{Lem7},  $P_{h+1}\bigcup P_{2h+1}\subseteq G$,
 which contradicts Claim 1.
Hence all blocks share  a common cut vertex of $G$, i.e.,  $G=L_{t,h}$, where $n=th+1$.
Moreover,  $P_h\bigcup P_{2h+1}\subseteq G $,  $P_{h-1}\bigcup P_h \bigcup P_{h+2}\subseteq G$, and
$P_h\bigcup P_h \bigcup P_{h+1}\subseteq G$.
If $\sum _{i=1}^{k}a_i\geq b_1+2$, then
 $$h=\sum \limits_{i=1}^{k}a_i+b_1-1\geq 2b_1+1,$$
 and
  $$2h+1= \sum \limits_{i=1}^{k}2a_i+2b_1-1\geq\sum \limits_{i=1}^{k}2a_i.$$
Hence $F \subseteq G$, a contradiction.
 If  $\sum _{i=1}^{k}a_i\leq b_1-1$, then
 $$h=\sum \limits_{i=1}^{k}a_i+b_1-1\geq \sum \limits_{i=1}^{k}2a_i,$$
 and
  $$2h+1=\sum \limits_{i=1}^{k}2a_i+2b_1-1\geq 2b_1+1.$$
 Hence $F \subseteq G$, a contradiction.
 If $\sum _{i=1}^{k}a_i= b_1$ and $k\geq 2$, then
 $$h-1=\sum \limits_{i=1}^{k}2a_i-2\geq 2a_k,$$
 $$h=\sum \limits_{i=1}^{k}2a_i-1> \sum \limits_{i=1}^{k-1}2a_i,$$
 and
 $$h+2=\sum \limits_{i=1}^{k}a_i+b_1+1=2b_1+1.$$
 Hence $F \subseteq G$, a contradiction.
 If $\sum _{i=1}^{k}a_i= b_1+1$ and $k\geq 2$, then
 $$h=\sum \limits_{i=1}^{k}2a_i-2\geq \sum \limits_{i=1}^{k-1}2a_i,$$
 $$h=\sum \limits_{i=1}^{k}2a_i-2\geq 2a_k,$$
 and
 $$h+1=\sum \limits_{i=1}^{k}a_i+b_1=2b_1+1.$$
  Hence  $F \subseteq G$, a contradiction.
So we have $F\in \{P_{2b_1}\bigcup P_{2b_1+1}, P_{2b_1+2}\bigcup P_{2b_1+1}\}$ and $G=L_{t,h}$, where $n=th+1$.
\end{Proof}

\section{Proofs of Theorems~\ref{Thm5} and \ref{Thm2}}

%In order to prove Theorem~\ref{Thm5}, we first prove the following Lemmas.
We will use the next several lemmas in  Theorem~\ref{Thm5}.
\begin{Lemma}\label{101}
Let $G$ be a connected graph of order $n$.\\
(i). If   $\delta(G)\geq 1$ and  $n\geq6$, then $2P_3\subseteq G$, unless either $G=U_{3,1}$, or
$G\subseteq L_{t_1,t_2,1,2}$ for  $n=t_1+2t_2+1$.\\
  (ii). If    $\delta(G)\geq 2$ and $n\geq8$, then $P_5\bigcup P_3\subseteq G$, unless $G\subseteq S_{n,2}^+$,
 $G\subseteq H_n^1$, $G\subseteq H_n^2$, or $G=L_{t,2}$ for $n=2t+1$.\\
(iii).  If    $\delta(G)\geq 2$ and $n\geq8$,  then $P_2\bigcup 2P_3\subseteq G$, unless  either $G\subseteq S_{n,2}$, or $G=L_{t,2}$ for  $n=2t+1$.
\end{Lemma}

\begin{Proof}
($i$). Suppose  $2P_3\nsubseteq G$. Let $P_l$ be a longest path  in $G$, denoted by $v_1v_2 \cdots v_l$.
Since $n\geq6$ and  $2P_3\nsubseteq G$,  we have $3\leq l\leq 5$.
If $l=3$, then $G$ must be a star $K_{1, n-1}$, i.e., $G=L_{n-1,1}$. Since $n\geq6$ and  $2P_3\nsubseteq G$, if $l=4$ then
 $G\subseteq L_{n-3,1,1,2}$.
Next we assume that  $l=5$.
Since $2P_3\nsubseteq G$,   we have $N_{P_5}(u)\subseteq \{v_3\}$ for all $u\in V(G)\backslash V(P_5)$. Let $X\subseteq V(G)\backslash V(P_5)$ such that $N_{P_5}(u)=\{v_3\}$ for all $u\in X$ and $Y=V(G)\backslash (V(P_5)\bigcup X)$.
Since  $n\geq 6$ and $2P_3\nsubseteq G$,   we have $X\neq \varnothing $ and  $P_3\nsubseteq  G[X]$. So we can assume that $G[X]$ consists of $p$ disjoint edges and $q$ isolated vertices, i.e.,
 $u_1u_2,\ldots,u_{2p-1}u_{2p},w_1,\ldots,w_q$, where $p+q\geq1$.
Since $P_5$ is a longest path in $G$, $Y$ must be  an independent set if $|Y|\geq1$.
Furthermore,  $N_G(u)=\{w_i\}$ and $N_G(v)=\{w_j\}$ for any two distinct vertices $u,v\in Y$ (it also holds for $|Y|=1$), where $1\leq i,j\leq q$ and $i\neq j$.
 Moreover,      $v_1v_4$, $v_1v_5$,  $v_2v_5\notin E(G)$, otherwise $2P_3\subseteq G$, a contradiction. Since   $2P_3\nsubseteq G$,
 if $v_2v_4\in E(G)$ then  $v_1v_3\notin E(G)$, $|Y|=p=0$, and $q=1$, which implies that $G=U_{3,1}$.
If  $v_2v_4\notin E(G)$, then $G\subseteq L_{t_1,t_2,1,2}$, where $n=t_1+2t_2+1$. Thus the assertion holds.

 ($ii$).  Suppose  $P_5\bigcup P_3\nsubseteq G$. Let $C_l$ be a longest cycle  in $G$, denoted by $v_1v_2\cdots v_lv_1$, and  $U=V(G)\backslash V(C_l)$.
Since $C_l$ is a longest cycle, none of vertices in $U$ is adjacent to any two  consecutive vertices of $C_l$. Since $n\geq8$ and $P_5\bigcup P_3\nsubseteq G$, it follows that  $3\leq l\leq 6$.
 We consider the following four cases.

 \textbf{Case 1}: $l=3$. This implies that all cycles in $G$ are triangles. Since $\delta(G)\geq2$,  $G$ has at least  two triangles.
 Let $T_1$ and $T_2$ be two triangles in $G$ such that the longest path $P_k$, denoted by $u_1u_2\cdots u_k$, with
$V(P_k)\bigcap V(T_1)=\{u_1\}$ and $V(P_k)\bigcap V(T_2)=\{u_k\}$, is as long as possible.
Since $n\geq8$, $\delta(G)\geq2$, and $P_5\bigcup P_3\nsubseteq G$, it follows that  $1\leq k\leq2$.
More precisely, $k=1$ (Otherwise $k=2$. $P_5\bigcup P_3\nsubseteq G$  together with $\delta (G)\geq 2$ implies that
$d_G(w)=\{u_1,u_2\}$  for every $w\in V(G)\backslash \bigcup_{i=1}^2 V(T_i)$.  Then $|V(G)\backslash \bigcup_{i=1}^2 V(T_i)|=1$, i.e., $n=7$,
 which  contradicts that $n\geq 8$.).
So all triangles in $G$ share a  common vertex.  Since $\delta(G)\geq2$,   $G=L_{t,2}$ for $n=2t+1$.

 \textbf{Case 2}:  $l=4$. Let $P_k$, denoted by $u_1\cdots u_k$, be a longest path in $G[U]$.
Since $P_5\bigcup P_3\nsubseteq G$,  it follows that $1\leq k\leq 4$. we consider the following four subcases.

{ \bf Subcase 2.1:}  $k=4$. Since $P_5\bigcup P_3 \nsubseteq G$,  it follows that  $N_{C_4}(u_1)=N_{C_4}(u_4)=\varnothing$.
Since $G$ is connected,  either $N_{C_4}(u_2)\neq \varnothing$  or $N_{C_4}(u_3)\neq \varnothing$.
If $u_1u_4\in E(G)$, then  $P_5\bigcup P_3 \subseteq G$, a contradiction.
 Since $\delta(G)\geq2$, if $u_1u_4\notin E(G)$ then  $u_1u_3\in E(G)$ and $u_2u_4\in E(G)$. Thus $P_5\bigcup P_3\subseteq G$, a contradiction.

 { \bf Subcase 2.2:}  $k=3$. By  Lemma~\ref{Lemma19}~(ii), $N_{C_4}(u_1)=N_{C_4}(u_3)$. Since $C_4$ is a longest cycle,
$d_{C_4}(u_1)=d_{C_4}(u_3)=1$. Note that $n\geq 8$. We have $P_5\bigcup P_3\subseteq G$, a contradiction.

{ \bf Subcase 2.3:}  $k=2$.  Obviously,  $G[U]$ consists of $p$ disjoint edges and $q$ isolated vertices, i.e.,
  $u_1u_2, u_3u_4, \ldots, u_{2p-1}u_{2p}, w_1, \ldots, w_q$, where $p\geq1$, $q\geq0$ and $2p+q=n-4$.
  Since $C_4$ is a longest cycle,  Lemma~\ref{Lemma19} $(i)$ implies that $N_{C_4}(u_{2i-1})=N_{C_4}(u_{2i})$ and $d_{C_4}(u_{2i-1})=d_{C_4}(u_{2i})=1$
for $1\leq i\leq p$.   So we assume without loss of generality that $N_{C_4}(u_1)=N_{C_4}(u_2)=\{v_1\}$.
Since $C_4$ is a longest cycle  and $\delta(G)\geq2$,   it follows that $N_{C_4}(w_1)=\cdots=N_{C_4}(w_r)=\{v_1,v_3\}$ and $N_{C_4}(w_{r+1})=\ldots=N_{C_4}(w_{r+t})=\{v_2,v_4\}$,
 where $r+t=q$.
 Moreover, since $P_5\bigcup P_3\nsubseteq G$, if $p\geq2$ then $N_{C_4}(u_3)=N_{C_4}(u_4)=\{v_3\}$, $p=2$, and $t=0$.  Thus  $G\subseteq H_n^2$.
In addition, since  $P_5\bigcup P_3\nsubseteq G$ and $r+t=q\geq 2$, if $p=1$ then  $t=0$.
   Thus $G\subseteq H_n^1$.

{ \bf Subcase 2.4:}  $k=1$. Obviously, the graph $G[U]$ consists of $q$ isolated vertices, i.e., $w_1, \ldots, w_q$, where $q=n-4\geq4$.
Since $C_4$ is a longest  cycle and $\delta(G)\geq2$,   $N_{C_4}(w_1)=\cdots=N_{C_4}(w_r)=\{v_1,v_3\}$ and $N_{C_4}(w_{r+1})=\cdots=N_{C_4}(w_{r+t})=\{v_2,v_4\}$,
 where $r+t=q\geq4$.
We assume without loss of generality that $r\geq t$. It follows that  $t=0$, otherwise $P_5\bigcup P_3\subseteq G$, a contradiction.
Furthermore, since $P_5\bigcup P_3\nsubseteq G$, we also have  $v_2v_4\notin E(G)$.  Thus  $G\subseteq S_{n,2}$.

{ \bf Case 3:}  $l=5$. Since $P_5\bigcup P_3 \nsubseteq G$,  we have $P_3\nsubseteq G[U]$.
Furthermore, $U$ is an  independent set (Otherwise, the graph $G[U]$ contains an edge $u_1u_2$.
Since $C_5$ is a longest cycle, Lemma~\ref{Lemma19}~(i) implies that $N_{C_5}(u_1)=N_{C_5}(u_2)$ and $d_{C_5}(u_1)=d_{C_5}(u_2)=1$.
Note that $n\geq 8$ and  $\delta(G)\geq2$.  We have  $P_5\bigcup P_3\subseteq G$, a contradiction.).
Since $n\geq 8$ and  $\delta(G)\geq2$, the longest cycle $C_5$ together with $P_5\bigcup P_3 \nsubseteq G$
implies that $N_{C_5}(u_1)=N_{C_5}(u_2)$ and $d_{C_5}(u_1)=d_{C_5}(u_2)=2$
  for any two vertices $u_1,u_2\in U$. We assume without loss of generality that $N_{C_5}(u)=\{v_1,v_3\}$ for all $u\in U$.  Since $P_5\bigcup P_3 \nsubseteq G$,  $G[\{v_2,v_4,v_5\}]$
contains exactly one edge $v_4v_5$. Hence $G\subseteq S_{n,2}^+$. Thus the assertion holds.

{ \bf Case 4:}  $l=6$. Let $U=V(G)\backslash V(C_6)$.   Since $P_5\cup P_3\nsubseteq G$, we have $P_8\nsubseteq G$. By Lemma~\ref{Lemma18}~($iii$),
 $N_{C_6}(u)=N_{C_6}(v)$ for any two vertices $u,v\in U$
and $2\leq d_{C_6}(u)=d_{C_6}(v)\leq 3$.
Since $P_5\bigcup P_3\nsubseteq G$,  it follows that $ d_{C_6}(u)=d_{C_6}(v)=2$ and the distance along $C_6$ between two vertices in $N_{C_6}(u)$ is at least $3$. However, we also have
$P_5\bigcup P_3\subseteq G$, a contradiction.

$(iii)$. Suppose $P_2\bigcup 2P_3\nsubseteq G$. This implies  $P_5\bigcup P_3\nsubseteq G$.
By Lemma~\ref{101}~($ii$), either $G\subseteq S_{n,2}^+$, $G\subseteq H_n^1$, $G\subseteq H_n^2$, or $G=L_{t,2}$ for $n=2t+1$.
However,   if $G=S_{n,2}^+$, $G\subseteq H_n^1$, or $G\subseteq H_n^2$, then $P_2\bigcup 2P_3\subseteq G$, a contradiction.
Hence   either $G\subseteq S_{n,2}$, or $G=L_{t,2}$ for  $n=2t+1$. Thus the assertion holds.
\end{Proof}

\begin{Lemma}\label{lem31}
Let $F=(\bigcup_{i=1}^k P_{2a_i})\bigcup (\bigcup_{i=1}^2 P_{2b_i+1})$, $ h=\sum _{i=1}^{k}a_i+\sum _{i=1}^2 b_i-1\geq3$,
 and $G$ be a connected graph of order $n$ with a longest cycle $C_l$, where $k\geq0$, $a_1\geq \dots \geq a_{k}\geq1$,  $b_1\geq b_2\geq1$,  and $n\geq 2h+4$.
If $\delta(G)\ge h$ and $F\nsubseteq G$, then  $2h\le l\le 2h+1$.
\end{Lemma}
\begin{Proof}
Let $C_l=v_1v_2\cdots v_lv_1$.  By Lemma~\ref{Lem9},  $l\geq2h$.
Since $\sum_{i=1}^{k}2a_i+\sum _{i=1}^2 (2b_i+1)=2h+4$ and $F\nsubseteq G$,  we have $P_{2h+4}\nsubseteq G$.
Note that $G$ is connected and $n\geq 2h+4$.  So  $l\le 2h+2$.    We now prove $l\le 2h+1$.  Suppose that $l=2h+2$.
Let  $U=V(G)\setminus V(C_{2h+2})$. By Lemma~\ref{Lemma18} ($iii$),
all vertices in $U$ have the same $C_{2h+2}$-neighborhood,  denoted by $V_1$, and $h\leq|V_1|\leq h+1$.
  Let    $V_2=V(C_{2h+2})\setminus V_1$. We consider  the following two cases.

{\bf Case 1:}  $|V_1|=h+1$. Since $C_{2h+2}$ is a longest cycle, none of  vertices in $U$ is adjacent to any two consecutive vertices in $V(C_{2h+2})$.
 So we assume without loss of generality  that $V_1=\{v_1, v_3, \ldots, v_{2h+1}\}$.
Let $H$, i.e., $H=(X,Y;E)$, be a bipartite subgraph of $G$,  where $X=V_1\setminus\{v_{2h+1}\}$ and $Y\subseteq U$ with $|Y|= h+2$ and $|U|= |V(G)|-|V(C_{2h+2})|\ge h+2 $.
By  definition, $H$ must be a complete bipartite graph.
Moreover, let $G_1$ be the graph obtained from $H$ by identifying $v_{2h-1}$ with a path $P_4$, $v_{2h-1}v_{2h}v_{2h+1}v_{2h+2}$. By  Lemma~\ref{Lemma20} $(ii)$, we have
$F\subseteq G$, a contradiction.

{\bf Case 2:}  $|V_1|=h$. Since none of  vertices in $U$ is  adjacent to any two consecutive vertices in $V(C_{2h+2})$, there exist two edges  $e_1,e_2\in  E(C_{2h+2})$ whose
 end vertices are all in $V_2$. If $e_1$ and $e_2$ share a common vertex,
then  we assume without loss of generality that $e_1=v_{2h}v_{2h+1}$ and $e_2=v_{2h+1}v_{2h+2}$.
This implies that $V_1=\{v_1,\ldots,v_{2h-1}\}$ and  $v_{2h}v_{2h+1}v_{2h+2}\subseteq G[V_2]$. By Lemma~\ref{Lemma20}~$(ii)$,  $F\subseteq G$, a contradiction.
So  $e_1$ and $e_2$ share no vertex in common.
 We assume  without loss of generality that $e_1=v_1v_2$ and $e_2=v_{2s+2}v_{2s+3}$, where $1\le s\le h-1$.
 Let $W=V_2\backslash \{v_1,v_2,v_{2s+2},v_{2s+3}\}$.
 Furthermore, we claim that $W$ is an independent set. In fact, if $v_pv_q\in E(G)$, where $v_p, v_q\in W$, then
 $uv_{p-1}v_{p-2}\cdots v_qv_pv_{p+1}\cdots v_{q-1}u$ is a longer cycle than $C_{2h+2}$ with $u\in U$, a contradiction.
Moreover, since $F\nsubseteq G$, Lemma~\ref{Lemma20} $(ii)$ implies that $G[V_2]$ contains exactly  two disjoint edges $e_1$ and $e_2$.
Thus for every $u\in \{v_1,v_2,v_{2s+2},v_{2s+3}\}$,   $N_G(u)\bigcap V_1 \geq h-1\geq2$
and $|N_G(v_1)\bigcap N_G(v_{2s+2})\bigcap V_1|\geq h-2\geq1$. Choose $v\in N_G(v_1)\bigcap N_G(v_{2s+2})\bigcap V_1$.
Note that %$G$ contains a path
$v_2v_1vv_{2s+2}v_{2s+3} \subseteq G$ % as subgraph
 and $v_2$ has a neighbor different from $v$ in $V_1$.
By Lemma~\ref{Lemma20} $(iii)$,  $F\subseteq G$, a contradiction.

Hence the assertion holds.
\end{Proof}

\begin{Lemma}\label{lem32}
Let $F=\bigcup_{i=1}^2 P_{2b_i+1}$, $ h=\sum _{i=1}^2 b_i-1\geq3$ and $G$ be a connected graph of order $n$  with a longest cycle $C_{2h}$, where $b_1\geq b_2\geq1$ and $n\geq 2h+4$.
If $\delta(G)\ge h$  and $F\nsubseteq G$, then  $P_3\nsubseteq G[U]$, where $U=V(G)\setminus V(C_{2h})$. \end{Lemma}
\begin{Proof}
Let $C_{2h}=v_1v_2\cdots v_{2h}v_1$.  Since $\sum _{i=1}^2 (2 b_i+1)=2h+4$ and $F\nsubseteq G$, we have $P_{2h+4}\nsubseteq G$.
 We first prove the following two Claims.

 {\bf Claim 1: }  $P_4\nsubseteq G[U]$.

Suppose $P_4\subseteq G[U]$.    Since $P_{2h+4} \nsubseteq G$, there exists one path $P_4$, denoted by  $u_1u_2u_3u_4 $,
such that $N_{C_{2h}}(u_1)=N_{C_{2h}}(u_4)=\varnothing$, and  either $N_{C_{2h}}(u_2)\neq \varnothing$  or $N_{C_{2h}}(u_3)\neq \varnothing$.
Furthermore,  since $P_{2h+4} \nsubseteq G$,   either $N_G(u_1)=\{u_2,u_3\}$ or $N_G(u_4)=\{u_2, u_3\}$. So $d_G(u_1)=2$ or $d_G(u_4)=2$,
which contradicts that $\delta(G)\geq h\geq3$.  So Claim 1 holds.

{\bf Claim 2:}  $P_3\nsubseteq G[U]$.

Suppose   $P_3\subseteq G[U]$. Let $P_3=u_1u_2u_3\subseteq G[U]$.  By Claim 1,  $P_4\nsubseteq G[U]$, and thus
 by Lemma~\ref{Lemma19} $(ii)$,  $N_{C_{2h}}(u_1)=N_{C_{2h}}(u_3)$.
  Moreover, $d_{C_{2h}}(u_1)\geq d_G(u_1)-2\geq h-2$ and  $d_{C_{2h}}(u_3)\geq d_G(u_3)-2\geq h-2$.
   Since $C_{2h}$ is a longest cycle in $G$, the distance along $C_{2h}$ between any two vertices in $N_{C_{2h}}(u_1)$  is at least 4,
 which implies that $4(h-2)\leq 2h$.  Thus $3\leq h\leq 4$.    We consider the following two cases.

{\bf Case 1:}   $h=4$.  Obviously, either $F=P_9\bigcup P_3$ or $F=P_7\bigcup P_5$.
It is easy to see that $d_{C_8}(u_1)=d_{C_8}(u_3)=2$, $d_G(u_1)=d_G(u_3)=4$, and the distance along $C_8$ between  two vertices in $N_{C_8}(u_1)$ is exactly $4$. Hence we assume  without loss of generality that $N_{C_8}(u_1)=N_{C_8}(u_3)=\{v_1,v_5\}$.   Since $P_4\nsubseteq G[U]$,  we have $u_1u_3\in E(G)$.
Note that $n\geq12$ in this  case.  So   $U\backslash \{u_1,u_2,u_3\}\neq \varnothing$.  Since $F\nsubseteq G$ and
$P_4 \nsubseteq G[U]$,  $u$ is  adjacent to none of vertices in $V(C_8)\bigcup\{u_1, u_2, u_3\}$ for all $u\in U\backslash \{u_1,u_2,u_3\}$,
 which contradicts that $G$ is connected.

 {\bf Case 2:}  $h=3$. Obviously, either $F=P_7\bigcup P_3$ or $ F=2P_5$. It is easy to see that $d_{C_6}(u_1)=d_{C_6}(u_3)=1$ and $d_G(u_1)=d_G(u_3)=3$.
We assume without loss of generality   that $N_{C_6}(u_1)=N_{C_6}(u_3)=\{v_1\}$.     Note that $P_4\nsubseteq G[U]$.  We have $u_1u_3\in E(G)$.
 Note that $n\geq10$ in this case. So $ U\backslash \{u_1,u_2,u_3\}\neq \varnothing$.
Furthermore, since $F\nsubseteq G$,  it follows that  $N_C(u)\bigcap \{v_1,v_2,v_3,v_5,v_6\}=\varnothing$ for all $u\in U\backslash \{u_1,u_2,u_3\}$.
Hence  $N_{C_6}(u)\subseteq \{v_4\}$. Moreover, since $P_4\nsubseteq G[U]$,  $u$ is not adjacent to any vertex in $\{u_1,u_2,u_3\}$.
  Thus  the connectedness of $G$ implies that  there exists $u_0\in U\backslash \{u_1,u_2,u_3\}$ such that $N_{C_6}(u_0)=\{v_4\}$.
However, since $\delta(G)\geq3$, there exists another vertex $w\in  U\backslash \{u_1,u_2,u_3\}$ such that $u_0w\in E(G)$. Therefore $F\subseteq G$, a contradiction.

Thus the assertion holds.
\end{Proof}

\begin{Corollary}\label{cor31}
Let $F=\bigcup_{i=1}^2 P_{2b_i+1}$, $ h=\sum _{i=1}^2 b_i-1\geq3$, and $G$ be a connected graph of order $n$  with a longest cycle $C_{2h}$, where $ b_1\geq b_2\geq1$ and $n\geq 2h+4$.
If  $\delta(G)\ge h$, then $F\subseteq G$, unless either  $G\subseteq S_{n,h}$, or $F=P_7\bigcup P_3$ and $G\subseteq K_2\bigvee \frac{n-2}{2}K_2$ for
even $n$.
\end{Corollary}
\begin{proof}
 Suppose $F\nsubseteq G$.   Let  $C_{2h}=v_1v_2\cdots v_{2h}v_1$ and $U=V(G)\setminus V(C_{2h})$.
 By Lemma~\ref{lem31},  $P_3 \nsubseteq G[U]$.  If  $U$ is an independent set, then  Lemma~\ref{Lemma18} $(i)$
 implies $G\subseteq S_{n,h}$.
  Hence  we assume that $G[U]$ consists of  $p$ disjoint edges and $q$ isolated vertices, i.e.,   $u_1u_2, u_3u_4, \ldots, $
 $u_{2p-1}u_{2p}, w_1, \ldots, w_q$, where $p\ge 1$ and $q\ge 0$.  By   Lemma~\ref{Lemma19} $(i)$,   $N_{C_{2h}}(u_{2i-1})=N_{C_{2h}}(u_{2i})$
 for $1\leq i\leq p$.   Furthermore,  $d_{C_{2h}}(u_{j})=d_G(u_{j})-1\geq h-1$ for $1\leq j\leq 2p$.
Since $C_{2h}$ is a longest cycle, the distance along $C_{2h}$ between any two vertices in $N_{C_{2h}}(u_1)$ is at least 3, which implies that $3(h-1)\leq2h$.
 Thus $h=3$. It follows that either $F=P_7\bigcup P_3$ or $F=2P_5$, and  $d_{C_6}(u_j)=2$ for $1\leq j\leq 2p$.
Moreover, the distance along $C_6$ between  two vertices in $N_{C_6}(u_1)$ is exactly 3.
Hence we assume without loss of generality that $N_{C_6}(u_1)=N_{C_6}(u_2)=\{v_1, v_4\}$.  Since  $F\nsubseteq G$, it follows that
 $N_{C_6}(u_{2i-1})=N_{C_6}(u_{2i})=\{v_1, v_4\}$ for $1\leq i\leq p$.  Next we prove that $q=0$. In fact, since $n\geq10$,
if $ q\geq1$ then  either $p=1$ and $q\ge 2$, or  $p\ge 2$ and $q\ge 1$.
  If $p=1$ and $q\ge 2$, then $d_G(w_i)=d_{C_6}(w_i)=3$ for $1\leq i\leq 2$ and thus $F\subseteq G$, a contradiction.
If $p\ge 2$ and $q\ge 1$, then $d_G(w_1)=d_{C_6}(w_1)=3$,
and thus  $F\subseteq G$, a contradiction.
So $n$ must be even. In addition,  since $F\nsubseteq G$, $G[\{v_2, v_3, v_5, v_6\}]$ consists of two disjoint edges $v_2v_3$ and $v_5v_6$.
 Therefore $G\subseteq K_2\bigvee \frac{n-2}{2}K_{2}$. Moreover, $F=P_7\bigcup P_3$.
Thus the assertion holds.
\end{proof}

\begin{Corollary}\label{cor32}
Let $F=(\bigcup_{i=1}^kP_{2a_i})\bigcup(\bigcup_{i=1}^2 P_{2b_i+1})$, $ h=\sum_{i=1}^ka_i+\sum _{i=1}^2 b_i-1\geq3$, and $G$ be a connected graph of order $n$ with a longest cycle $C_{2h}$, where
 $k\ge 1$, $b_1\geq b_2\geq1$, and $n \geq 2h+4$.
If  $\delta(G)\ge h$, then $F\subseteq G$ unless,  either  $G\subseteq S_{n,h}$, or $F=P_4\bigcup 2P_3$ and $G\subseteq K_2\bigvee \frac{n-3}{2}K_2$ for  even   $n$.
\end{Corollary}
\begin{proof}
Suppose $F\nsubseteq G$.
Let $F'=\bigcup_{i=1}^2 P_{2b_i'+1}$, where  $b_1'=\sum_{i=1}^k a_i+b_1$ and  $b_2'=b_2$. We have $ F'\nsubseteq G $.
By Corollary~\ref{cor31}, either $G\subseteq S_{n,h}$, or $F'=P_7\bigcup P_3$ and $G\subseteq K_2\bigvee \frac{n-2}{2}K_2$ for even  $n$.
Note  $F\nsubseteq G$ whenever $G\subseteq S_{n,h}$. So we  assume that $F'=P_7\bigcup P_3$ and $G\subseteq K_2\bigvee \frac{n-2}{2}K_2$ for even  $n$.
Then we have $F\in\{P_4\bigcup 2P_3,P_2\bigcup P_5\bigcup P_3, 2P_2\bigcup 2P_3\}$.
In addition,  $ P_2\bigcup P_5\bigcup P_3\subseteq G$, $2P_2\bigcup 2P_3\subseteq G$, and $P_4\bigcup 2P_3\nsubseteq G$.
Hence $F=P_4\bigcup 2P_3$.
Thus the assertion holds.
\end{proof}

\begin{Lemma}\label{Lem33}
Let $F=\bigcup_{i=1}^2 P_{2b_i+1}$, $ h=\sum _{i=1}^2 b_i-1\geq3$, and $G$ be a connected graph of order $n$  with a longest cycle $C_{2h+1}$,
where  $b_1\geq b_2\geq1$ and $n\geq4(2h+1)^2\binom{2h+1}{h}$.
If $\delta(G)\ge h$, then $F\subseteq G$, unless  either  $G\subseteq S^+_{n,h}$, or $F=P_9\bigcup P_3$ and
$ G\subseteq K_3\bigvee \frac{n-3}{2}K_2$ for odd $n$.
\end{Lemma}

\begin{Proof}
Suppose  $F\nsubseteq G$.   Since $\sum_{i=1}^2 (2b_i+1)=2h+4$,  we have $P_{2h+4}\nsubseteq G$.
   Let   $C_{2h+1}=v_1v_2\cdots v_{2h+1}v_1$ and $U=V(G)\setminus V(C_{2h+1})$.
We consider the  following two cases.

{\bf Case 1:} $U$ is an independent set.  Since $C_{2h+1}$ is a longest cycle, none of vertices in $U$ is adjacent to any two consecutive vertices along $C_{2h+1}$.
It follows that $d_{C_{2h+1}}(u)= h$ for all $u\in U$. By Lemma~\ref{Lem5}, there exists $V_1\subseteq V(C_{2h+1})$ and
$U_1\subseteq U$ with $|V_1|=h$ and $|U_1|=h+2$ such that $H(V_1,U_1;E)$ is a complete bipartite graph. Hence   $N_{C_{2h+1}}(u)=V_1$ for all $u\in U_1$.
We assume   without loss of generality that $V_1=\{v_2, v_4,\ldots, v_{2h}\}$.
Moreover,  $N_{C_{2h+1}}(u)=V_1$ for all $u\in U\backslash U_1$. Otherwise, there exists  $u\in U\backslash U_1$ such that
$N_{C_{2h+1}}(u)\bigcap (V(C_{2h+1})\backslash V_1)\neq \varnothing$.
Since $u$ is not adjacent to    any two consecutive vertices along $C_{2h+1}$,  we have $v_1$ or  $v_{2h+1}\in N_{C_{2h+1}}(u)\bigcap (V(C_{2h+1})\backslash V_1)$.
Note   $v_{2},v_{2h}\in V_1$ and $v_1v_{2h+1}\in E(G)$.   By Lemma~\ref{Lemma20} $(ii)$,  $F\subseteq G$, a contradiction.
Furthermore, $\{v_1,v_3,\ldots, v_{2h-1}\}$ is an independent set (Othewise, let $v_{2s+1}v_{2t+1}\in E(G)$ and  $u\in U$,
 where $1\leq 2s+1<2t+1\leq 2h-1$.  It follows that $uv_{2s}v_{2s-1} \cdots v_2v_1v_{2h+1}v_{2h}\cdots v_{2t+2}v_{2t+1}v_{2s+1}v_{2s+2}\cdots v_{2t-1} v_{2t}u$ is a longer cycle than $C_{2h+1}$, a contradiction).
Similarly, $\{v_3,\ldots, v_{2h+1}\}$ is an independent set. Hence $G[\{v_1,v_3,\ldots,v_{2h+1}\}]$ contains exactly one edge $v_1v_{2h+1}$.
Thus $G\subseteq S^+_{n,h}$.

{\bf Case 2:} $G[U]$ contains at least one edge. Since  $\delta(G)\geq 2$, we have $P_3 \nsubseteq G[U]$.  Hence
we assume that $G[U]$ consists of $p$ disjoint edges and $q$ isolated vertices, i.e., $u_1u_2, u_3u_4, \ldots, u_{2p-1}u_{2p}, w_1, \ldots, w_q$, where $p\ge 1$ and $q\ge 0$.
By  Lemma~\ref{Lemma19} $(i)$,  $N_{C_{2h+1}}(u_{2i-1})=N_{C_{2h+1}}(u_{2i})$ for $1\leq i\leq p$.
 Furthermore, $d_{C_{2h+1}}(u_{j})=d_G(u_{j})-1\geq h-1$ for $1\leq j\leq 2p$.
Since $C_{2h+1}$ is a longest cycle in $G$, the distance along $C_{2h+1}$ between any two vertices in $N_{C_{2h+1}}(u_1)$ is at least 3, which implies
 that $3(h-1)\leq 2h+1$. Thus $3\leq h\leq 4$.
We claim that $h=4$ (Otherwise $h=3$ and thus either $F=P_7\bigcup P_3$ or $F=2P_5$.
 Moreover, the distance along $C_7$ between  two vertices in $N_{C_7}(u_1)$ is exactly 3.
 So  we assume without loss of generality that
 $N_{C_7}(u_1)=N_{C_7}(u_2)=\{v_1, v_4\}$.  Since $F\nsubseteq G$,  it follows that $N_{C_7}(u)\bigcap \{v_1,v_2,v_3,v_4,v_5,v_7\}=\varnothing$
 for all $u\in U\backslash \{u_1,u_2\}$.
 Hence $N_{C_7}(u)\subseteq\{v_6\}$,
 which contradicts  $d_{C_7}(u)\geq2$.).
Hence either $F=P_9\bigcup P_3$ or $ F=P_7\bigcup P_5$.  Moreover,  the distance along $C_9$ between any two vertices in $N_{C_9}(u_1)$ is exactly 3.
So we assume  without loss of generality   that $N_{C_9}(u_1)=N_{C_9}(u_2)=\{v_1, v_4,v_7\}$.
  Since $F\nsubseteq G$, it follows that  $N_{C_9}(u_{2i-1})=N_{C_9}(u_{2i})=\{v_1, v_4,v_7\}$ for $1\leq i \leq p$, and $q=0$.
 So $n$ must be odd. In addition, since  $F\nsubseteq G$, $G[\{v_2, v_3, v_5, v_6,v_8,v_9\}]$ consists of  three disjoint edges $v_2v_3$,
 $v_5v_6$ and $v_8v_9$. Therefore $G=K_3\bigvee \frac{n-3}{2}K_{2}$. Moreover, $F=P_9\bigcup P_3$.
Thus the assertion holds.
\end{Proof}

\begin{Lemma}\label{Lem34}
Let $F=(\bigcup_{i=1}^kP_{2a_i})\bigcup(\bigcup_{i=1}^2 P_{2b_i+1})$, $ h=\sum_{i=1}^ka_i+\sum _{i=1}^2 b_i-1\geq3$, and $G$ be a connected graph of order $n$ with a longest cycle $C_{2h+1}$, where
$k\ge1$, $b_1\geq b_2\geq1$, and $n\geq4(2h+1)^2\binom{2h+1}{h}$.
 If   $\delta(G)\ge h$, then $F\subseteq G$, unless
 $F=P_6\bigcup 2P_3$ and $ G\subseteq K_3\bigvee \frac{n-3}{2}K_2$ for odd $n$.
\end{Lemma}

\begin{proof}
Suppose $F\nsubseteq G$.
Let  $F'=\bigcup_{i=1}^2 P_{2b_i'+1}$, where $b_1'=\sum_{i=1}^k a_i+b_1$ and $b_2'=b_2$. We have $ F'\nsubseteq G $.
By Lemma~\ref{Lem33},   either $G\subseteq S_{n,h}^+$,  or  $F'=P_9\bigcup P_3$ and $G\subseteq K_3\bigvee \frac{n-3}{2}K_2$ for odd $n$.
If $G\subseteq S_{n,h}^+$, then $F\subseteq G$, a contradiction.
 Thus we assume that $F'=P_9\bigcup P_3$ and $G\subseteq K_3\bigvee \frac{n-3}{2}K_2$ for odd $n$.
Hence $F\in \{P_2\bigcup P_7\bigcup P_3, P_2\bigcup 2P_5, P_4\bigcup P_5\bigcup P_3,
P_6\bigcup 2P_3,
2P_2\bigcup P_5\bigcup P_3,
 P_4\bigcup P_2\bigcup 2P_3,\\
 3P_2\bigcup 2P_3\}$.
In addition, since $F\nsubseteq G$,  it follows that $F=P_6\bigcup 2P_3$.
Thus the assertion holds.
\end{proof}

  Theorem~\ref{Thm5} can be stated as follows.

 \begin{Theorem}\label{Thm5-rep}
Let $F=(\bigcup_{i=1}^k P_{2a_i})\bigcup (\bigcup_{i=1}^2 P_{2b_i+1})$, $ h=\sum _{i=1}^{k}a_i+\sum _{i=1}^2 b_i-1\geq2$,
and $G$ be a 2-connected graph of order $n$,
where $k\geq0$, $a_1\geq \dots \geq a_{k}\geq1$, $b_1\geq b_2$, and $n\geq4(2h+1)^2\binom{2h+1}{h}$.

(a). If  $\delta(G)\geq h$ and $k=0$, then $F\subseteq G$, unless one of the following holds:\\
 (i). $G\subseteq S^+_{n,h}$;\\
  (ii). $F=P_7\bigcup P_3$  and $G\subseteq K_2\bigvee \frac{n-2}{2}K_2$, where $n$ is even; \\
   (iii).    $F=P_9\bigcup P_3$ and $G\subseteq K_3\bigvee \frac{n-3}{2}K_2$, where $n$ is odd.

(b). If $\delta(G)\geq h$ and  $k\geq1$, then $F\subseteq G$, unless one of the following holds:\\
(iv). $G\subseteq S_{n,h}$;\\
  (v). $F= P_4\bigcup 2P_3$ and $G\subseteq K_2\bigvee \frac{n-2}{2}K_2$, where $n$ is even ;\\
   (vi). $F=P_6\bigcup 2P_3$ and $G\subseteq K_3\bigvee \frac{n-3}{2}K_2$, where  $n$ is odd.
\end{Theorem}

\begin{Proof}
%(a). If $h=1$, then $F=2P_3$. By  Lemma~\ref{101} $(i)$,  we have $F\subseteq G$
%unless  either $G=U_{3,1}$, or $G\subseteq L_{t_1,t_2,1,2}$, where $n=t_1+2t_2+1$.
%However, $G$ is not 2-connected whenever  $G=U_{3,1}$ or $G\subseteq L_{t_1,t_2,1,2}$. Hence $F\subseteq G$.
Suppose $F\nsubseteq G$. If $h=2$, then $F=P_5\bigcup P_3$.  By Lemma~\ref{101} $(ii)$,  either $G\subseteq S_{n,2}^+$,
$G\subseteq H_n^1$, $G\subseteq H_n^2$, or $G=L_{t,2}$ for $n=2t-1$.
However, $G$ is not 2-connected whenever   $G\subseteq H_n^1$, $G\subseteq H_n^2$, or $G=L_{t,2}$. Hence  $G\subseteq S_{n,2}^+$.
Next we assume that $h\geq3$.      Let $C_l$ be the longest cycle in $G$.   By Lemma~\ref{lem31}, $2h\leq l\leq 2h+1$.
It is easy to see that Corollary~\ref{cor31} together with Lemma~\ref{Lem33}  implies that  $G\subseteq S_{n,h}^+$,
 $F= P_7\bigcup P_3$ and $G\subseteq K_2\bigvee \frac{n-2}{2}K_2$ for  even $n$,
 or  $F=P_9\bigcup P_3$ and $G\subseteq K_3\bigvee \frac{n-3}{2}K_2$ for odd $n$.

(b).  Suppose  $F\nsubseteq G$. If $h=2$, then $F=P_2\bigcup 2P_3$.
 By Lemma~\ref{101} $(iii)$,   either $G\subseteq S_{n,2}$, or  $G=L_{t,2}$ for $n=2t-1$.
However, if   $G=L_{t,2}$, then $G$ is not 2-connected.   Hence    $G\subseteq S_{n,2}$.
Next we assume that $h\geq3$.       Let $C_l$ be the longest cycle in $G$.       By Lemma~\ref{lem31}, $2h\leq l\leq 2h+1$.
It is easy to see that  Corollary~\ref{cor32} together with Lemma~\ref{Lem34} implies that
 $G\subseteq S_{n,h}$, $F= P_4\bigcup 2P_3$ and $G\subseteq K_2\bigvee \frac{n-2}{2}K_2$ for even $n$,
 or $F=P_6\bigcup 2P_3$ and $G\subseteq K_3\bigvee \frac{n-3}{2}K_2$ for odd $n$. Thus the assertion holds.
   \end{Proof}

We will use the next two lemmas in  Theorem~\ref{Thm2}.

\begin{Lemma}\label{Lem102}
Let $F=\bigcup_{ i=1}^2 P_{2b_i+1}$, $ h=\sum _{i=1}^2 b_i-1\geq3$, and $G$ be a connected  graph of order $n$  with at least one cut vertex,  where
$b_1\geq b_2\geq1$ and $n\geq2h+4$.
  If   $\delta(G)\geq h$, then $F\subseteq G$, unless one of the following holds:\\
 (i)  $F=P_{2b_1+1}\bigcup P_{2b_1-1}$ and $G=L_{t,h}$, where $n=th+1$;\\
  (ii) $F=2P_{2b_1+1}$ and $G=U_{3,h}$, where $n=3h+3$;\\
 (iii) $F=2P_{2b_1+1}$ and $G\subseteq L_{t_1,t_2,h,h+1}$, where   $n=t_1h+t_2(h+1)+1$;\\
 (iv)  $F=2P_{2b_1+1}$ and $ G\subseteq F_{t_1,t_2,h,h+1}$, where  $n=t_1h+(t_2+1)(h+1)+1$;\\
 (v)  $F=2P_{2b_1+1}$ and $G\subseteq T_{t_1,t_2,h,h+1}$, where  $n=t_1h+(t_2+2)(h+1)+1$.
\end{Lemma}

\begin{Proof}
Suppose  $F\nsubseteq G$. Since $\sum_{i=1}^{k}2a_i+\sum_{i=1}^2 (2b_i+1)=2h+4$,   we have $P_{2h+4}\nsubseteq G$.
Note that $h+2\geq 2b_2+1$ and $2h+1\geq2b_1+1$.  We have $P_{h+2}\bigcup P_{2h+1}\nsubseteq G$.
Note that $h+1\geq 2b_2+1$ and $2h+1\geq 2b_1+1$ for $b_1\geq b_2+1$.  We  also have $P_{h+1}\bigcup P_{2h+1}\nsubseteq G$ for $b_1\geq b_2+1$.
 Note that $h\geq2b_2+1$ and $2h+1\geq 2b_1+1$ for $b_1\geq b_2+2$. Moreover,  $ P_h\bigcup P_{2h+1}\nsubseteq G$  for $b_1\geq b_2+2$.

 Since $G$ has at least one cut vertex,  it  has  at least two end blocks.
Furthermore, since  $\delta(G)\geq h$,  it follows that $|V(B)|\geq h+1$  for every end block  $B$.
Choose  two end blocks  $B_1$ and $B_2$  of $G$   such that $p(B_1,B_2)$ is as large as possible, where $p(B_1,B_2)$ is
the order of a longest path between  the cut vertex $u_1$ of $G$ in $V(B_1)$ and the cut vertex $u_2$ of $G$ in $V(B_2)$.
 Lemma~\ref{Lem7} together with $P_{2h+4}\nsubseteq G$ implies that $p(B_1,B_2)\leq 3$. We assume without loss of generality that $|V(B_1)|\leq|V(B_2)|$.
We consider the following three cases.

\textbf{Case 1}: $p(B_1,B_2)=1$. This implies that  all blocks in $G$ are end blocks and they share a common cut vertex of $G$.
We first prove the following Claim.

{\bf Claim:} $G$ has at least three end blocks.

Suppose that   $G$ has exactly two end blocks $B_1$ and $B_2$.
Since $n\geq 2h+4$ and  $|V(B_1)|\leq|V(B_2)|$, if $|V(B_1)|\geq h+2$ then  $|V(B_2)|\geq h+3$.
By Lemma~\ref{Lem7},  $P_{2h+4}\subseteq G$, a contradiction.
Hence we assume that $|V(B_1)|=h+1$  and thus $|V(B_2)|\geq h+4$.
We claim  $h=3$ (Otherwise $h\geq 4$. By Lemma~\ref{Lem7}, $B_2$ has a path $P_{\mathrm{min}\{2h,h+4\}}$ with an end vertex $u_2$.  Since $\mathrm{min}\{2h,h+4\}\geq h+4$, we have $P_{2h+4}\subseteq G$, a contradiction.).
It follows that either $F=P_7\bigcup P_3$ or $F=2P_5$. Furthermore, by Lemmas~\ref{Lem7} and \ref{Lemma16} $(iii)$,  $F\subseteq G$, a contradiction.
This completes the claim.

By Claim above,  $G$ has at least three end blocks.  Since  $P_{h+2}\bigcup P_{2h+1}\nsubseteq G$, all blocks have order at most $h+2$ and
$P_h\bigcup P_{2h+1}\subseteq G$. Note  $P_h\bigcup P_{2h+1}\nsubseteq G$ for  $b_1\geq b_2+2$. Hence   $b_2\leq b_1\leq b_2+1$.
If $b_1=b_2$, then $F=2P_{2b_1+1}$ and $G\subseteq L_{t_1,t_2,h,h+1}$, where $n=t_1h+t_2(h+1)+1$.
Next we assume that $b_1=b_2+1$. Note   $P_{h+1}\bigcup P_{2h+1}\nsubseteq G$ for $b_1=b_2+1$.
It follows that  all  blocks have order $h+1$.
Thus $F=P_{2b_1+1}\bigcup P_{2b_1-1}$ and $G=L_{t,h}$, where $n=th+1$.

\textbf{Case 2}: $p(B_1,B_2)=2$.  This implies that all blocks in $G$ are end blocks.   Lemma~\ref{Lem7} together with $P_{2h+4}\nsubseteq G$
 implies that $|V(B_1)|=h+1$ and $h+1\leq |V(B_2)|\leq h+2$.
Thus $G$ is a  graph obtained by adding an edge to the centers of $L_{t_1,t_2,h,h+1}$
 and $L_{t_3,t_4,h,h+1}$, where $n=(t_1+t_3)h+(t_2+t_4)(h+1)+2$, $t_1+t_2\geq1$, and $t_3+t_4\geq1$.
Furthermore, since  $P_{h+2}\bigcup P_{2h+1}\nsubseteq G$,  either $t_1+t_2=1$ or $t_3+t_4=1$.
We assume without loss of generality that $t_1+t_2=1$. Since  $n\geq2h+4$,  it follows that $t_3+t_4\geq2$.
So $P_{h+1}\bigcup P_{2h+1}\subseteq G$.
Note  $P_{h+1}\bigcup P_{2h+1}\nsubseteq G$ for $b_1\geq b_2+1$. Then   $b_1=b_2$  and thus $F=2P_{2b_1}$.
Moreover, since $P_{h+2}\bigcup P_{2h+1}\nsubseteq G$,  it follows that $t_2=0$ and thus $t_1=1$.
Therefore, $F=2P_{2b_1+1}$ and  $G\subseteq F_{t_3,t_4,h,h+1}$,  where $n=t_3h+(t_4+1)(h+1)+1$.

\textbf{Case 3}:  $p(B_1,B_2)=3$.  Lemma~\ref{Lem7} together with $P_{2h+4}\nsubseteq G$ implies that  $|V(B_1)|=|V(B_2)|=h+1$.
Note  $P_{h+2}\bigcup P_{2h+1}\nsubseteq G$.  The end block  $B_i$ is the unique end block with $u_i\in V(B_i)$ for $1\leq i\leq2$.
 Denote $u_1uu_2$ by the path between $u_1$ and $u_2$.
Since $\delta(G)\geq 3$,  there exists an end block of order at least $h+1$ with the vertex $u$.
Hence $P_{h+1}\bigcup P_{2h+1}\subseteq G$.  Note   $P_{h+1}\bigcup P_{2h+1}\nsubseteq G$ for $b_1\geq b_2+1$. So $b_1=b_2$ and thus $F=2P_{2b_1}$.
Furthermore, since $p(B_1,B_2)=3$, there is exactly one path $P_3$ between $u_1$ and $u_2$.
Since $P_{h+2}\bigcup P_{2h+1}\nsubseteq G$, if $u_1u_2\in E(G)$ then there is exactly one end block with the vertex $u$ and it has order $h+1$.
Hence $F=2P_{2b_1+1}$ and  $G=U_{3,h}$,  where $n=3h+3$.
Next we assume  $u_1u_2\notin E(G)$.  Obviously, $u_1u$ and $u_2u$ are both cut edges of $G$.
Since $P_{h+2}\bigcup P_{2h+1}\nsubseteq G$, all blocks that contains  $u$ have order at most $h+2$.
Therefore  $F=2P_{2b_1+1}$ and $G\subseteq T_{t_1,t_2,h,h+1}$, where $n=t_1h+(t_2+2)(h+1)+1$.

Thus the assertion holds.
\end{Proof}

\begin{Lemma}\label{Lem103}
Let $F=(\bigcup_{i=1}^k P_{2a_i})\bigcup (\bigcup_{ i=1}^2 P_{2b_i+1})$, $ h=\sum_{i=1}^{k}a_i+\sum _{i=1}^2 b_i-1\geq2$,
and $G$ be a connected graph of order $n$  with at least one cut vertex,
where $k\geq1$, $a_1\geq \cdots \geq a_{k}\geq1$, $b_1\geq b_2\geq3$,  and  $n\geq2h+4$.
If $\delta(G)\geq h$, then $F\subseteq G$, unless $F=P_2\bigcup 2P_{2b_1+1}$ and $G=L_{t,h}$, where  $n=th+1$.
\end{Lemma}

\begin{Proof}
 Suppose  $F\nsubseteq G$. If $h=2$, then $F=P_2\bigcup 2P_3$.  By Lemma~\ref{101} $(ii)$,  we have $G\subseteq S_{n,2}^+$, $G\subseteq H_n^1$,
 $G\subseteq H_n^2$, or $G=L_{t,2}$ for $n=2t+1$.  However,  $G$ is  2-connected whenever  $G\subseteq S_{n,2}^+$ and  $F\subseteq G$ whenever
$G\subseteq H_n^1$ or  $G\subseteq H_n^2$.  Hence  $G=L_{t,2}$, where  $n=2t+1$. Next we assume that $h\geq3$. Let   $F'=\bigcup_{i=1}^2 P_{2b_i'+1}$, where $b_1'=\sum_{i=1}^k a_i+b_1$ and $b_2'=b_2$.  We have $F'\nsubseteq G $.
Note that $k\geq1$ and $b_1\geq b_2$. It follows that $b_1'>b_2'$.
By Lemma~\ref{Lem102},   $G=L_{t,h}$, where $n=th+1$ and $F'=P_{2b_1'+1}\bigcup P_{2b_1'-1}$.
Moreover, since $F'=P_{2b_1'+1}\bigcup P_{2b_1'-1}$, it follows that $k=1$, $a_1=1$, and $b_1=b_2$.
Hence $F=P_2\bigcup 2P_{2b_1+1}$. Thus the assertion holds.
\end{Proof}

  Theorem~\ref{Thm2} can be stated as follows.

\begin{Theorem}\label{Thm2-rep}
Let $F=(\bigcup_{i=1}^k P_{2a_i})\bigcup (\bigcup_{ i=1}^2 P_{2b_i+1})$, $ h=\sum_{i=1}^{k}a_i+\sum _{i=1}^2 b_i-1$,
and  $G$ be a connected  graph of order $n$  with at least one cut vertex,  where
$a_1\geq \dots \geq a_{k}\geq1$, $b_1\geq b_2\geq1$, $k\geq0$, and $n\geq2h+4$.

(a).  If   $\delta(G)\geq h\geq1$ and  $k=0$, then $F\subseteq G$, unless one of the following holds:\\
(i)  $F=P_5\bigcup P_3$ and either $G\subseteq H_n^1$ or $G\subseteq H_n^2$;\\
 (ii) $F=P_{2b_1+1}\bigcup P_{2b_1-1}$ and $G=L_{t,h}$, where $n=th+1$;\\
  (iii) $F=2P_{2b_1+1}$ and $G=U_{3,h}$, where $n=3h+3$;\\
 (iv) $F=2P_{2b_1+1}$ and $G\subseteq L_{t_1,t_2,h,h+1}$, where   $n=t_1h+t_2(h+1)+1$;\\
 (v) $F=2P_{2b_1+1}$ and  $ G\subseteq F_{t_1,t_2,h,h+1}$, where  $n=t_1h+(t_2+1)(h+1)+1$;\\
 (\romannumeral6)   $F=2P_{2b_1+1}$ and $G\subseteq T_{t_1,t_2,h,h+1}$, where  $n=t_1h+(t_2+2)(h+1)+1$.

 (b). If $\delta(G)\geq h\geq2$ an $k\geq1$, then $F\subseteq G$, unless  $F=P_2\bigcup 2P_{2b_1+1}$ and $G=L_{t,h}$, where  $n=th+1$.
\end{Theorem}

\begin{Proof}
(a). Suppose  $F\nsubseteq G$. If $h=1$, then $F=2P_3$. By Lemma~\ref{Lem33} ($i$),  either $G=U_{3,1}$ or
$G\subseteq L_{t_1,t_2,1,2}$ for $n=t_1+2t_2+1$.
Next we assume that $h=2$. It follows that  $F=P_5\bigcup P_3$. By Lemma~\ref{Lem33} ($ii$), $G\subseteq S_{n,2}^+$,
 $G\subseteq H_n^1$, $G\subseteq H_n^2$, or $G=L_{t,2}$ for $n=2t+1$.
However, $G$ is 2-connected whenever  $G\subseteq S_{n,2}^+$. Hence  $G\subseteq H_n^1$, $G\subseteq H_n^2$, or $G=L_{t,2}$ for $n=2t+1$.
Finally we assume that $h\geq 3$. The assertion follows  by Lemma~\ref{Lem102}.

(b). The assertion follows  by Lemma~\ref{Lem103}.
\end{Proof}

\indent{\bf Acknowledgements:}

  The authors would like to thank the anonymous referee for many helpful and constructive suggestions to an earlier version of this paper, which results in an great improvement.


\begin{thebibliography}{12}

\bibitem{A} B. Andrasfai, Paths, Circuits, and Loops of Graphs, (Hungarian) Mat. Lapok 13 (1962) 95-107.

\bibitem{AS}  A.A. Ali and W. Staton, On extremal graphs with no long paths, Electron. J. Combin. 3 (1996),  Research Paper 20, approx. 4 pp.

\bibitem{BK}   N. Bushaw and N. Kettle, Tur\'{a}n numbers of multiple paths and equibipartite forests, Combin. Probab. Comput. 20 (2011)  837--853.

\bibitem{CZ2018} M.-Z. Chen and X.-D. Zhang, The number of edges, spectral radius and Hamilton-connectedness of graphs, J. Comb. Optim. 35 (2018) 1104--1127.

\bibitem{Dirac}   G.A. Dirac,  Some theorems on abstract graphs,  Proc. London Math Soc. 2 (1952) 69--81.


 \bibitem{EG}  P. Erd\H{o}s and  T. Gallai, On maximal paths and circuits of graphs,  Acta Math. Acad. Sci. Hungar. 10 (1959) 337--356.

 \bibitem{FKV} Z. F\"{u}redi, A. Kostochka, J. Verstra\"{e}te, Stability in the Erd\H{o}s--Gallai Theorems on cycles and paths, J. Combin. Theory Ser. B, 121 (2016) 197-228.

% \bibitem{Gorgol}  I. Gorgol, Tur\'{a}n number for disjoint copies of graphs, Graphs Combin. 27 (2011) 661-667.

 \bibitem{LLP}  B. Lidick\'{y}, H. Liu, C. Palmer, On the Tur\'{a}n number of forests, Electron. J. Combin. 20 (2) (2013) Paper 62, 13 pp.


  \bibitem{NY} V. Nikiforov and X.Y. Yuan,  Maxima of the $Q$-index: graphs without long paths,  Electron. J. Linear Algebra 27 (2014) 504--514.


 \bibitem{West} D.B. West, Introduction to Graph Theory, Prentice-Hall, London, 2001.
\bibitem{Yuan} L.-T.~Yuan and X.-D~Zhang, The Tur\'{a}n number of  disjoint copies of path,  Discrete Math. 340 (2) (2017) 132-139.

\end{thebibliography}
\end{document}